\pdfoutput=1
\documentclass{amsart}

\usepackage[english]{babel}
\usepackage[utf8]{inputenc}
\usepackage[T1]{fontenc}

\usepackage[a4paper,top=3cm,bottom=2cm,left=3cm,right=3cm,marginparwidth=1.75cm]{geometry}

    \usepackage{amsthm}
    \usepackage{amssymb}
    \usepackage{amsmath}
    \usepackage{mathtools}
    \usepackage{epsfig}
    \usepackage{caption}
    \usepackage[colorinlistoftodos]{todonotes}
    \usepackage[colorlinks=true, allcolors=blue]{hyperref}
    \usepackage{xcolor}
    \usepackage{pinlabel}
    \usepackage{algorithm}
    \usepackage{algpseudocode}
    \usepackage{relsize}
    \usepackage{geometry}
    \usepackage{subcaption}
    \usepackage[numbers]{natbib}
    \usepackage{hyperref}
    \usepackage{longtable}
    \usepackage[shortlabels]{enumitem}

\theoremstyle{definition}
\newtheorem{definition}{Definition}[section]

\theoremstyle{plain}
\newtheorem{theorem}[definition]{Theorem}
\newtheorem{lemma}[definition]{Lemma}
\newtheorem{proposition}[definition]{Proposition}

\newtheorem{remark}[definition]{Remark}

\theoremstyle{remark}

\begin{document}

\title{The mixed Hilden braid group and the plat equivalence in  handlebodies} 

\author[P. Cavicchioli]{Paolo Cavicchioli}
\address{Dipartimento di Matematica,
Università di Bologna, Piazza di Porta S. Donato, 5, 40126, Bologna, Italy. }
\email{paolo.cavicchioli@unibo.it}
\urladdr{https://www.unibo.it/sitoweb/paolo.cavicchioli}

\author[S. Lambropoulou]{Sofia Lambropoulou}
\address{School of Applied Mathematical and Physical Sciences, 
National Technical University of Athens, 
Zografou Campus, 9 Iroon Polytechneiou st., 15772 Athens, Greece.}
\email{sofia@math.ntua.gr}
\urladdr{http://www.math.ntua.gr/~sofia}




\begin{abstract}
Given a knot or link in the handlebody, $H_g$, of genus $g$ we prove that it can always be represented as the plat closure of a braid in $H_g$.  We further establish the Hilden braid group for the handlebody, as a subgroup of the mixed braid group, whose elements have the first $g$ strands fixed forming the identity braid. We then formulate and prove the algebraic equivalence connecting mixed plats belonging to the same link isotopy class in $H_g$. The plat closure representation can be particularly suitable for computing knot invariants. 
\end{abstract}

\maketitle

\section*{Introduction} 
The representation of links as plat closures of braids with an even number of strands dates back to the works of Hilden \cite{hilden1975generators} and of Birman \cite{birman1976stable}, who in 1976 proved that any link in $\mathbb R^3$ or $S^3$, the 3-sphere,  could be represented as the plat closure of a braid in some braid group \(\mathcal{B}_{2n}\). 
As an example, recall that all rational knots/links (also known as 2-bridge knots/links) can be represented as the plat closure by a braid with 4 strands, leaving the last strand as the identity strand \cite{bankwitz1934viergeflechte}. Further, in \cite{hilden1975generators} Hilden described the plat isotopy of braids in the same braid group \(\mathcal{B}_{2n}\), by means of a subgroup \(K_{2n}\) of \(\mathcal{B}_{2n}\), the Hilden braid group. In \cite{birman1976stable} Birman extended the work of Hilden by the stabilization moves, to provide the equivalence relation for braids having isotopic plat closures. This braid equivalence is analogous to the classical Markov equivalence for braids with isotopic standard closures \cite{markov1935freie}. The Markov theorem together with the Alexander braiding theorem \cite{alexander1923lemma} furnish the algebraic theory for the isotopy equivalence of knots and links represented as standard closures of braids. 

Since then, a lot of work has been done using the plat representation. For instance, in \cite{bigelow2002homological} Bigelow studied the possibility of computing the Jones polynomial of a link, represented as plat closure of a braid, in terms of the action of that braid over a homological pairing defined on a covering of the configuration space of \(n\) points into the \(2n\)-punctured disc. This result was analyzed in recent years also from a computational point of view \cite{jordan2008estimating, Garnerone2007QuantumAB} and in order to compute other link invariants, such as the Conway potential function \cite{yun2011braid} or other polynomial invariants \cite{garnerone2006quantum}. 

The above led many researchers to generalize the results obtained in \(\mathbb{R}^3\) in the more general setting of c.c.o. 3-manifolds and handlebodies. For instance, using Heegaard surfaces, in \cite{doll1993generalization} Doll introduced the notion of $(g,b)$-decomposition or generalized bridge decomposition for links in a closed, connected and orientable (from now on c.c.o.) \(3\)-manifold, opening the way to the study of links in 3-manifolds via surface braid groups and their plat closures, see works of Bellingeri and Cattabriga \cite{bellingeri2012hilden} and Cattabriga and Gabrovšek  \cite{cattabriga2018markov}.

 In a parallel development, in  \cite{lambropoulou1997markov}, Lambropoulou and Rourke established the analogues of the Alexander and Markov theorems for knots and links in c.c.o. 3-manifolds, representing them first as mixed links in $S^3$ and then  as standard closures of mixed  braids. For the equivalence they used the concept of the $L$-move. In \cite{lambropoulou2000braid}  the mixed braid group  \(\mathcal{B}_{g,n}\) is introduced and studied, which is the subgroup of the Artin braid group   \(\mathcal{B}_{g+n}\) that  contains the identity braid $I_g$ on \(g\) strands as a fixed subbraid, enabling the  algebraic formulation of the Markov theorem in c.c.o. 3-manifolds \cite{lambropoulou2006algebraic}. 
Adapting the work in \cite{lambropoulou1997markov} to the setting of  a genus $g$ handlebody, $H_g$, in \cite{lambropoulou2000braid} H\"aring and Lambropoulou \cite{haring2002knot}, established the analogues of the Alexander and Markov theorems for knots and links in $H_g$, by representing them as standard closures of mixed braids, and utilizing the mixed braid groups and results from  \cite{lambropoulou2006algebraic}. see also \cite{sossinsky1992preparation}.

The goal of this article is to establish the counterpart plat closure theory for  handlebodies. We first show that it is always possible to represent a link in the handlebody $H_g$ of genus $g$ via the plat closure of a braid in $H_g$ (Theorem~\ref{teo_braiding_plat_hand}). Then we establish the Hilden braid group for $H_g$ as the Hilden braid subgroup, $K_{g,2n}$, of the mixed braid group \(\mathcal{B}_{g,2n}\), and we provide a set of generators for $K_{g,2n}$ (Theorem~\ref{theo:hilden_mixed}). We finally formulate and prove the algebraic equivalence connecting two plats in $H_g$ belonging to the same isotopy class (Theorem~\ref{teo:main_equivalence}). We are especially interested in this last part because it would be the starting point of the construction and computation of polynomial link invariants for links in the setting of handlebodies, and then extend further to the case of other 3-manifolds. With the present paper we fill in a gap in the literature, namely to study the plat closure for braids in a handlebody and formulate and prove their isotopy equivalence using the mixed Hilden braid group.  

We also studied the problem of links in handlebodies  in \cite{cavicchioli2023passing} from a different aspect: we proved that given the representation of a knot or link via the plat closure of a braid in \(\mathbb{R}^3\), in a handlebody or  in a thickened surface, it is possible to create algorithmically another braid representing the same link type but with the standard  closure, and vice versa. In this paper we establish that every knot or link in $H_g$ has a plat representative.

In Section \ref{Preliminaries} we  recall the notions of interest in the case of \(\mathbb{R}^3\). Namely, the Artin braid group \(\mathcal{B}_{m}\), the existence of standard and plat closure of braids, the braidings of knots and links to plats and to standardly closed braids, as well as the corresponding  equivalence relations reflecting isotopy. In particular, in Section \ref{sec:standard_Hilden} we recall the definition of the Hilden braid group used in the plat equivalence. 
 In Section \ref{Handlebody} we recall the representation of the handlebody $H_g$ in \(\mathbb{R}^3\) by the identity braid $I_g$ and of links and braids in $H_g$ by mixed links and mixed braids, and the definition and a presentation of the mixed braid group \(\mathcal{B}_{g, m}\).  
 In Section \ref{sec:plat_closure_mixed} we first recall the standard closure of mixed braids, the mixed braiding and the
mixed braid equivalence (analogues of the Alexander and Markov theorems for $H_g$).   We then define the plat closure of braids in $H_g$ and of mixed braids, and we state and prove the first part of the main result, the existence of a representation of links in handlebodies via plat closure of braids in Section \ref{Main_result}. 
 In Section \ref{Sec:Mixed_hilden_group} we define the Hilden braid group for $H_g$ and the mixed Hilden braid group, as the Hilden subgroup of the mixed braid group. Then, starting from a set of generators of the classical Hilden braid group, we obtain a set of generators for the mixed Hilden braid group, which is a crucial ingredient in proving the equivalence theorem. 
 Finally, in Section \ref{Equivalence} we state and prove the main theorem of the paper, the equivalence theorem for links in handlebodies represented via plat closed braids.

\smallbreak
This research was partially funded by the National Technical University of Athens and by the University of Bologna.

\section{Preliminaries: the case of $\mathbb{R}^3$} \label{Preliminaries}

In this section we first recall classical braids and their group structure. Then their standard closure,  the Alexander theorem for braiding any oriented link,  as well as the braid equivalence corresponding to classical link isotopy (the Markov theorem and its $L$-move analogue). We then proceed with recalling the plat closure for braids and the plat braiding. Finally, we recall the Hilden braid subgroup of the braid group and the plat closure braid equivalence corresponding to classical link isotopy, all by following the work of Birman on plats \cite{birman1976stable}. All of this will be our basic tools to prove the equivalence theorem in the setting of handlebodies. 

\subsection{The Artin braid group}

The Artin braid group \cite{artin1947theory}, \(\mathcal{B}_{n}\),  is defined as the fundamental group of the unordered configuration space of \(n\) points on the real plane, i.e. a group which elements are equivalence classes of \(n\) paths \(\Psi = (\psi_1, \dots, \psi_n), \psi_i : [0,1] \rightarrow \mathbb{R}^2\) going from a set of \(n\) points of the plane to themselves, and never intersecting (\(\psi_i(t) \neq \psi_j(t), \forall t \in [0,1], i,j\in {1, \dots, n}, i<j\)). 

The classical presentation of \(B_n\) is generated by elements \(\sigma_i, i\in {1, \dots, n-1}\), which exchanges strands \(i\) and \(i+1\) as in Fig.~\ref{fig:braid_sigma_i}, with relations: 
\begin{align*}
   \sigma_i \sigma_j & = \sigma_j \sigma_i, \qquad \text{if } |i-j|\geq 2\\
   \sigma_i \sigma_{i+1} \sigma_i & = \sigma_{i+1} \sigma_i \sigma_{i+1}. \\
\end{align*}
For further information concerning the braid group the reader can refer, for example, to \cite{birman2005braids, kassel2008braid}. 

\begin{figure}[H]
    \centering
    \includegraphics[width = .5\textwidth]{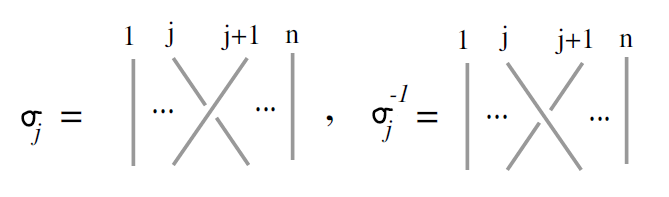}
    \caption{The \(\sigma_j\) and \(\sigma_j^{-1}\) generators.}
    \label{fig:braid_sigma_i}
\end{figure}

\subsection{The standard closure, the braiding and the braid equivalence }\label{Section:standard_braiding_equivalence}

In order for a braid to represent a link we have to `close' it, for example identifying the starting set of points on the plane with itself. This operation can be viewed as connecting the top endpoints to their corresponding bottom endpoints with simple unlinked arcs, like in the right part of Fig.~\ref{fig:standard_closure_plat}. This is called \emph{standard closure} of a braid. It is important to note that this closure naturally gives rise to an oriented link: indeed, if we consider the strands of the braid to be oriented downward, then all closing arcs will be oriented upward and the whole orientation is coherent. 

\begin{figure}[H]
    \centering
    \includegraphics[width = .85 \textwidth]{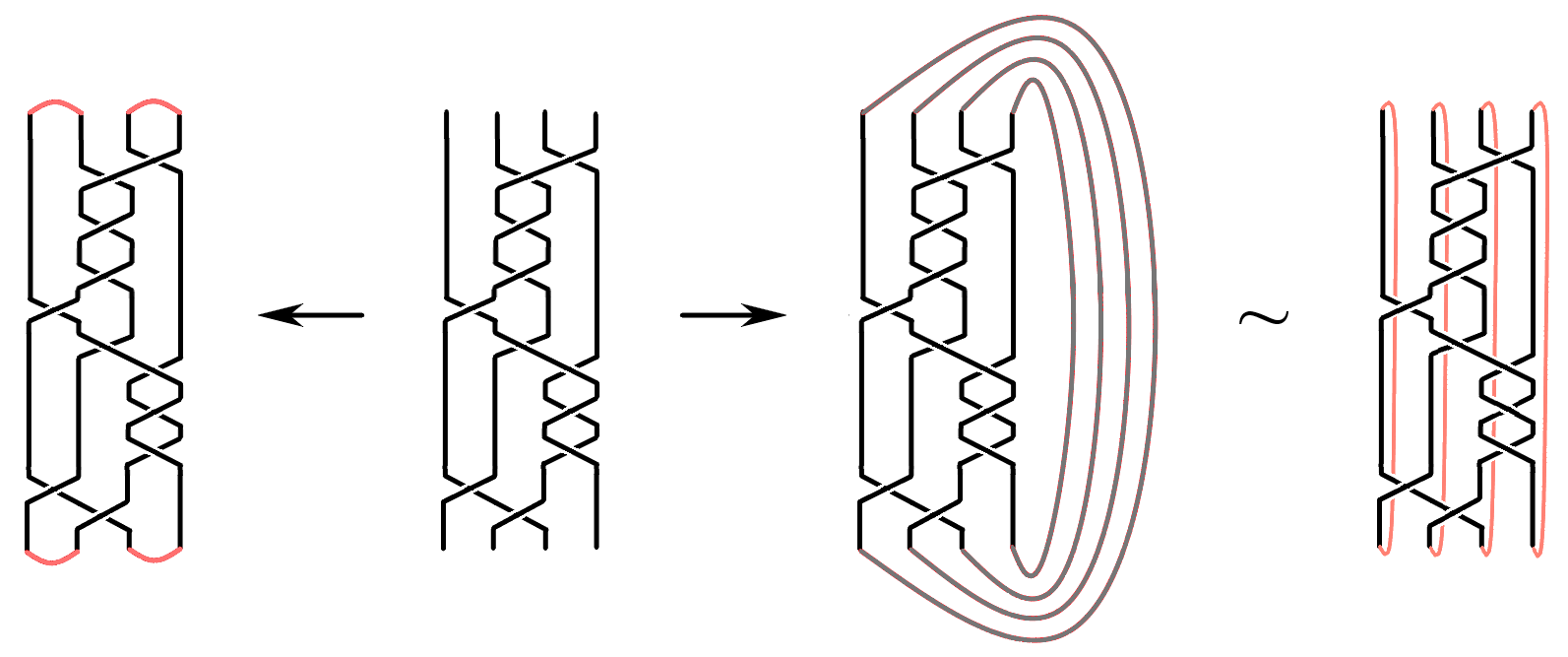}
    \caption{On the right, the standard closure of a braid and its isotopy to sliding the closing arcs to the back of the braid. On the left, the plat closure of the same braid.}
    \label{fig:standard_closure_plat}
\end{figure}

Conversely, by a result of H.~Brunn and of J.W.~Alexander \cite{alexander1923lemma, brunn1897verknotete}: 
\begin{theorem}[The Alexander theorem]
Every oriented link in \(\mathbb{R}^3\) can be represented as the standard closure of a braid in some \(\mathcal{B}_{n}\). 
\end{theorem}

Different braids (also with a different number of strands) can be closed to give the same (oriented) link up to isotopy. A theorem by A.A.~Markov \cite{markov1935freie} (with an improvement by N.~Weinberg \cite{weinberg1939equivalence}, reducing the original 3 moves to 2) gives us a way of determining whether closing two braids results in isotopic links or not: 

\begin{theorem}[The Markov theorem]\label{Thm:two_markov}
Let \(\beta_n \in \mathcal{B}_n, \beta_m \in \mathcal{B}_m\) be two braids, and let \(L_n, L_m\) be the two links obtained by the standard closure of \(\beta_n, \beta_m\) respectively. Then \(L_n\) is equivalent to \(L_m\) if and only if \(\beta_n\) can be obtained from \(\beta_m\) after a finite sequence of these two moves: 
\begin{itemize}
    \item Conjugation: \(\beta \sim \gamma \beta \gamma^{-1}, \quad \beta, \gamma \in \mathcal{B}_n\); 
    \item Stabilization move: \(\beta \sim \beta \sigma_n^{\pm 1} \in \mathcal{B}_{n+1}, \quad \beta \in \mathcal{B}_{n}\). 
\end{itemize}
\end{theorem}

\noindent Fig.~\ref{fig:markov_R3} illustrates abstractly the 2 moves of the theorem. In the figure  a red closing arc is indicated, showing the isotopy between the (standard) closures of the middle and right-hand braid. 
\begin{figure}[H]
    \centering
    \includegraphics[width = .6\textwidth]{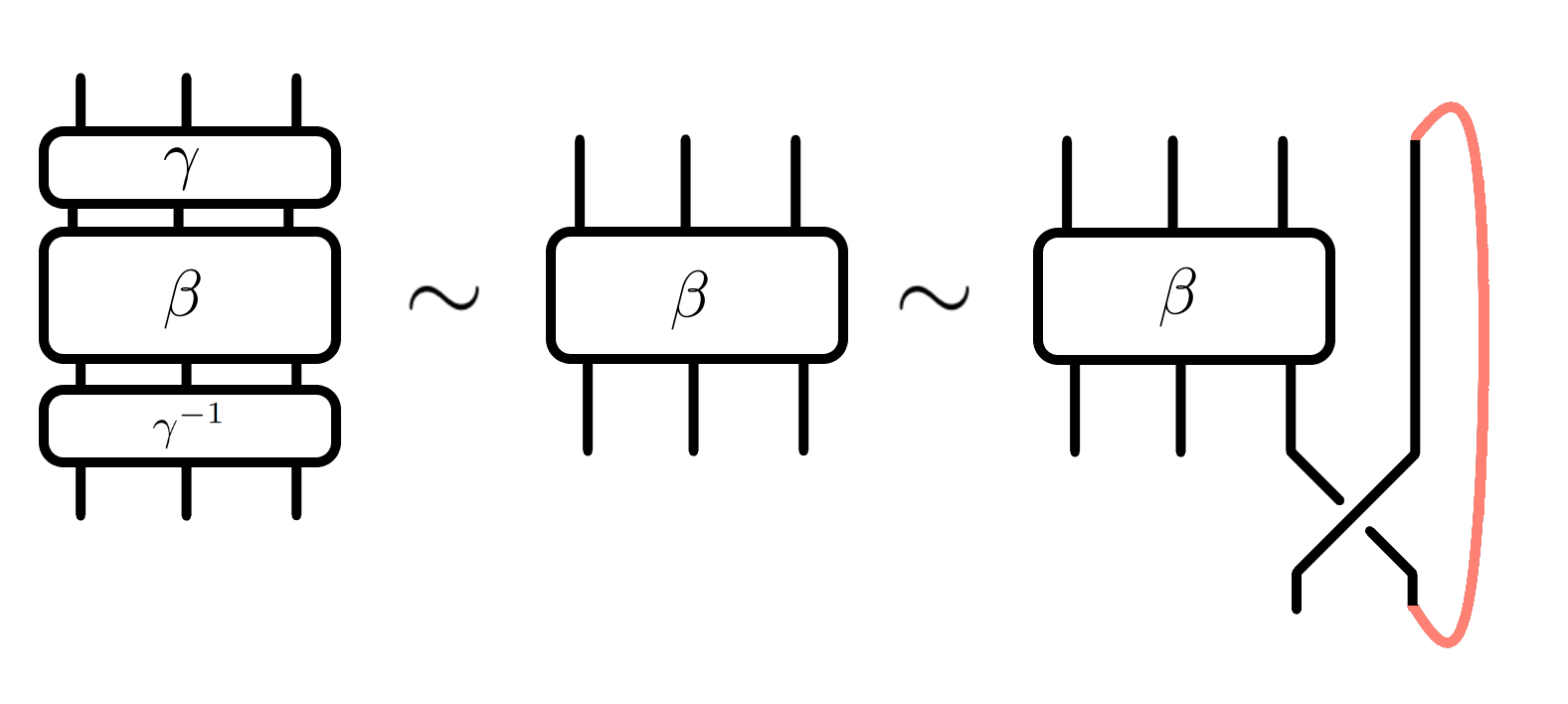}
    \caption{The conjugation (on the left) and the stabilization move (on the right).}
    \label{fig:markov_R3}
\end{figure}

Some years ago Lambropoulou and Rourke proved in \cite{lambropoulou1997markov} another version of the Markov theorem, with only one equivalence move: the {\it \(L\)-move}. This is defined as follows (see abstract illustration in Fig.~\ref{fig:L_move}): take a braid \(\beta \in \mathcal{B}_n\) and a point \(P\) in one of its strands and not aligned to any of the crossings or endpoints of \(\beta\); cut the arc at \(P\), bend the two resulting  arcs apart by a small isotopy and introduce two new vertical arcs (slightly tilted) to two new
top and bottom endpoints in the  vertical line of \(P\). These two new arcs are both oriented downward and they run either both \emph{under} all other arcs of the diagram (first instance in Fig.~\ref{fig:L_move}) or both \emph{over} (third instance in Fig.~\ref{fig:L_move}) all other arcs of the diagram, thus performing an \emph{under \(L\)-move} (or \emph{\(L_u\)-move}) and a \emph{over \(L\)-move} (or \emph{ \(L_o\)-move}) respectively. The right-hand illustration of the figure indicates an equivalent version of an \(L_o\)-move, where an in-box crossing is introduced, using braid isotopy. This last form allows to have algebraic expressions for the \(L\)-moves (see \cite{lambropoulou1997markov}). In the figure we have also indicated red closing arcs for the \(L\)-moves, showing that the (standard) closures of all braids are isotopic.

\begin{figure}[H]
    \centering
    \includegraphics[width = \textwidth]{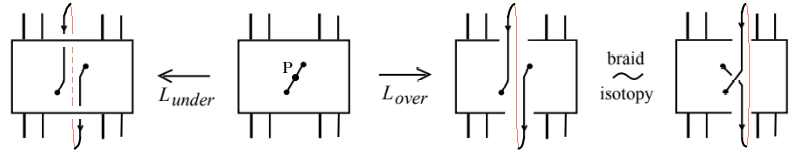}
    \caption{Performing an \(L\)-move. The two results are the \emph{under} and \emph{over} \(L\)-moves. The right-hand illustration shows the version of an \(L\)-move with an in-box crossing. }
    \label{fig:L_move}
\end{figure}

The new statement is the following: 
\begin{theorem}[\(L\)-move braid equivalence]\label{Thm:one_move_markov}
Let \(\beta_n \in \mathcal{B}_n, \beta_m \in \mathcal{B}_m\) be two braids, and let \(L_n, L_m\) be the two links obtained by the standard closure of \(\beta_n, \beta_m\) respectively. Then \(L_n\) is equivalent to \(L_m\) if and only if \(\beta_n\) can be obtained from \(\beta_m\) after a finite sequence of \(L\)-moves and braid isotopies. 
\end{theorem}
Note that the stabilization move of Theorem~\ref{Thm:two_markov} is a special case of an \(L\)-move. Further, from Theorem~\ref{Thm:one_move_markov} it follows that the braid conjugation of Theorem~\ref{Thm:two_markov} can be achieved by the \(L\)-moves. 

\subsection{The plat closure and representation of links as plat closures  of braids}

Let  \(B\) be a braid with an even number of strands, i.e.  \(B \in \mathcal{B}_{2n}\) some \(n\). By \cite{birman1976stable} one can close the open ends by connecting adjacent top or bottom endpoints using simple arcs (see left part of Fig.~\ref{fig:standard_closure_plat}). This procedure is called the \textit{plat closure} of the braid and the resulting link {\it plat}, denoted as  \(\overline{B}\). 

We can extend the notion of plat closure also to braids with an odd number of strands as follows: let \(\beta \in \mathcal{B}_n\), we consider \({B}'\) to be the image of \(\beta\) under the natural embedding of \(\mathcal{B}_n\) into \(\mathcal{B}_{n+1}\), and we define the plat closure of \(\beta\) as the plat closure of \({B}'\) with the procedure described above. In practice, we add an extra strand in the far right of the braid \(\beta\). 

It is important to note that, closing in the plat way does not keep the natural orientation of the braid. In fact, when we close the braid using those short arcs, we connect two downward strands. 

In \cite{birman1976stable}(Proposition 1) Birman proved that it is always possible to describe any link as the plat closure of a braid in \(\mathbb{R}^3\). This is the plat  analogue of the classical Alexander theorem that uses the standard closure of braids. The idea of Birman's algorithm is simple: in a projection plane equipped with top to bottom direction, consider a knot or link diagram with no horizontal arcs. Then pull by an isotopy every local maximum to the top and every local minimum to the bottom, as abstracted in Fig.~\ref{fig:braiding}. The result will be a plat. 

\begin{figure}[H]
    \centering
    \includegraphics[width = .7\textwidth]{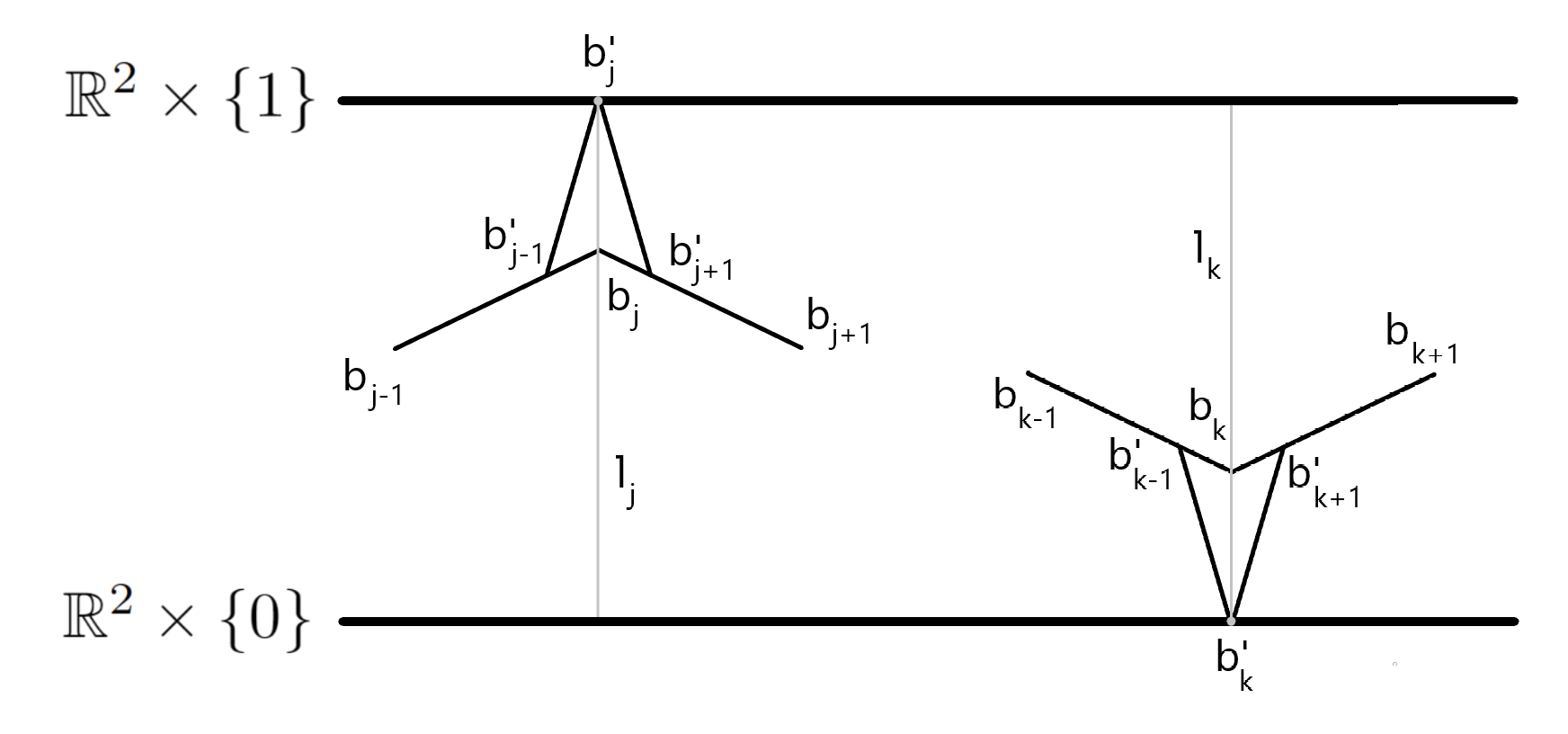}
    \caption{The braiding process. }
    \label{fig:braiding}
\end{figure}

In \cite{cavicchioli2023passing}  we remarked that, if one closes a braid via the classical standard closure and push by isotopy the closing arcs right next to the endpoints, then the resulting closed braid is a plat. See rightmost part of Fig.~\ref{fig:standard_closure_plat}. This observation  provides an alternative proof of the Birman result  (see \cite[Theorem 3.2]{cavicchioli2023passing}).

\subsection{The Hilden braid group}\label{sec:standard_Hilden}

Now we want to recall a subgroup of the braid group \(\mathcal{B}_{2n}\)  which will be useful for the construction of the equivalence between links represented by braids closed via plat closure.  Following the construction of \cite{hilden1975generators}, consider a braid as an embedding in \(\mathbb{R}^3\), such that the plat closure arcs of the bottom part of the braid lie at level \(\pi : \{z = 0\}\). This plane \(\pi\) cuts \(\mathbb{R}^3\) into two halves; let the upper half be \(\mathbb{R}^3_+\) and let the set of endpoints of the closure arcs in \(\pi\) be \(A = (a_{11}, a_{12}, \dots, a_{m1}, a_{m2})\), as in Fig.~\ref{fig:plat_closure_plane}. 

\begin{figure}[H]
    \centering
    \includegraphics[width = .55\textwidth]{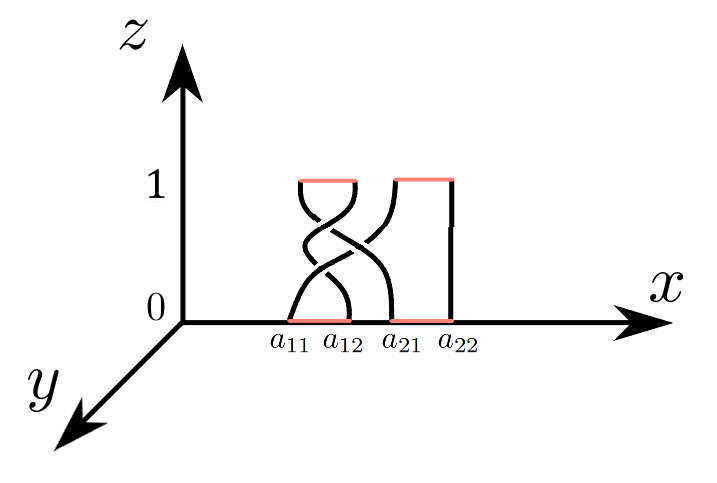}
    \caption{The braid embedded in \(\mathbb{R}^3\) such that the closure arcs lie on the plane \(\pi : z = 0\). }
    \label{fig:plat_closure_plane}
\end{figure}

Then the group of isotopy classes \(\Psi = [\psi]\) of orientation-preserving homeomorphisms \(\psi : (\pi, A) \rightarrow (\pi, A)\) is the classical braid group \(\mathcal{B}_{2n}\). 

\begin{definition}\label{def:Hilden_group}
The {\it Hilden braid group} \cite{hilden1975generators}, \(K_{2m}\), is defined as the subgroup of \(\mathcal{B}_{2m}\) generated by the equivalence classes of homeomorphisms of \(\mathbb{R}^3_+\) leaving the set of closure arcs invariant on its boundary.     
\end{definition}

In \cite{hilden1975generators} Hilden proves that \(K_{2m}\) is finitely generated and gives a set of generators in terms of elements of \(\mathcal{B}_{2m}\). Moreover, he points at the following reduced set of generators, which Birman \cite{birman1976stable}  proves is minimal. 

\begin{proposition}\label{prop:hilden_generators}
The Hilden braid group \(K_{2m}\) is generated by the set 
\[\{\sigma_{2i-1}, 1\leq i\leq m; \  \mu_i = \sigma_{2i} \sigma_{2i-1} \sigma_{2i+1} \sigma_{2i}, 1\leq i \leq m-1; \ \rho_{ij}, 1\leq i,j \leq m; \ \omega_{ij}, 1 \leq i,j \leq m\},\] 
with: 
\[\rho_{ij} = \left \{ \begin{array}{ll}
\vspace{4pt}
(\sigma_{2j} \sigma_{2j-1}) (\sigma_{2j+1} \sigma_{2j}) \dots (\sigma_{2i-2} \sigma_{2i-3}) \sigma_{2i-2} (\sigma_{2i-3} \sigma_{2i-2}) \sigma_{2i-3}^{-1} (\sigma_{2i-4}^{-1} \sigma_{2i-3}^{-1}) \dots (\sigma_{2j-1}^{-1} \sigma_{2j}^{-1}), \\
\vspace{8pt}
\text{ if } \ j<i \text{,}\\  
\vspace{4pt}
(\sigma_{2j-2}^{-1} \sigma_{2j-1}^{-1}) (\sigma_{2j-3}^{-1} \sigma_{2j-2}^{-1}) \dots (\sigma_{2i}^{-1} \sigma_{2i+1}^{-1}) \sigma_{2i}^{-1} (\sigma_{2i+1}^{-1} \sigma_{2i}^{-1}) \sigma_{2i+1}^{-1} (\sigma_{2i+2} \sigma_{2i+1}) \dots (\sigma_{2j-1} \sigma_{2j-2}), \\
\text{ if } \ i<j.
\end{array}
\right.\]

\[\omega_{ij} = \left \{ \begin{array}{ll}
\vspace{5pt}
(\sigma_{2j} \sigma_{2j-1}) (\sigma_{2j+1} \sigma_{2j}) \dots (\sigma_{2i-2} \sigma_{2i-3}) (\sigma_{2i-1} \sigma_{2i-2}) \sigma_{2i-1} (\sigma_{2i-2} \sigma_{2i-1}) (\sigma_{2i-3} \sigma_{2i-2}) \sigma_{2i-3}^{-1}  \\
\vspace{10pt}
(\sigma_{2i-4}^{-1} \sigma_{2i-3}^{-1}) \dots (\sigma_{2j-1}^{-1} \sigma_{2j}^{-1}) , \text{ if } \ j<i \text{,}\\
\vspace{5pt}
(\sigma_{2j-2}^{-1} \sigma_{2j-1}^{-1}) (\sigma_{2j-3}^{-1} \sigma_{2j-2}^{-1}) \dots (\sigma_{2i}^{-1} \sigma_{2i-1}^{-1}) (\sigma_{2i+1}^{-1} \sigma_{2i}^{-1}) \sigma_{2i-1} (\sigma_{2i}^{-1} \sigma_{2i+1}^{-1}) (\sigma_{2i-1}^{-1} \sigma_{2i}^{-1}) \sigma_{2i+1}^{-1} \\
(\sigma_{2i+2} \sigma_{2i+1}) \dots (\sigma_{2j-1} \sigma_{2j-2}), \text{ if } \ i<j.\end{array}
\right.\]

\noindent Furthermore, a minimal set of generators for \(K_{2m}\)  is the following: 
\[\{\sigma_1; \  \lambda_1 = \rho_{1,2} = \sigma_2 \sigma_1^2 \sigma_2; \   \mu_i = \sigma_{2i} \sigma_{2i-1} \sigma_{2i+1} \sigma_{2i}, 1\leq i \leq m-1\}.\] 
\end{proposition} 

A visual representation of \(\rho_{ij}\) and \(\omega_{ij}\) generators is depicted in Figure~\ref{fig:gen_hilden_rho_omega}. 

\begin{figure}[H]
    \centering
    \includegraphics[width = 1 \textwidth]{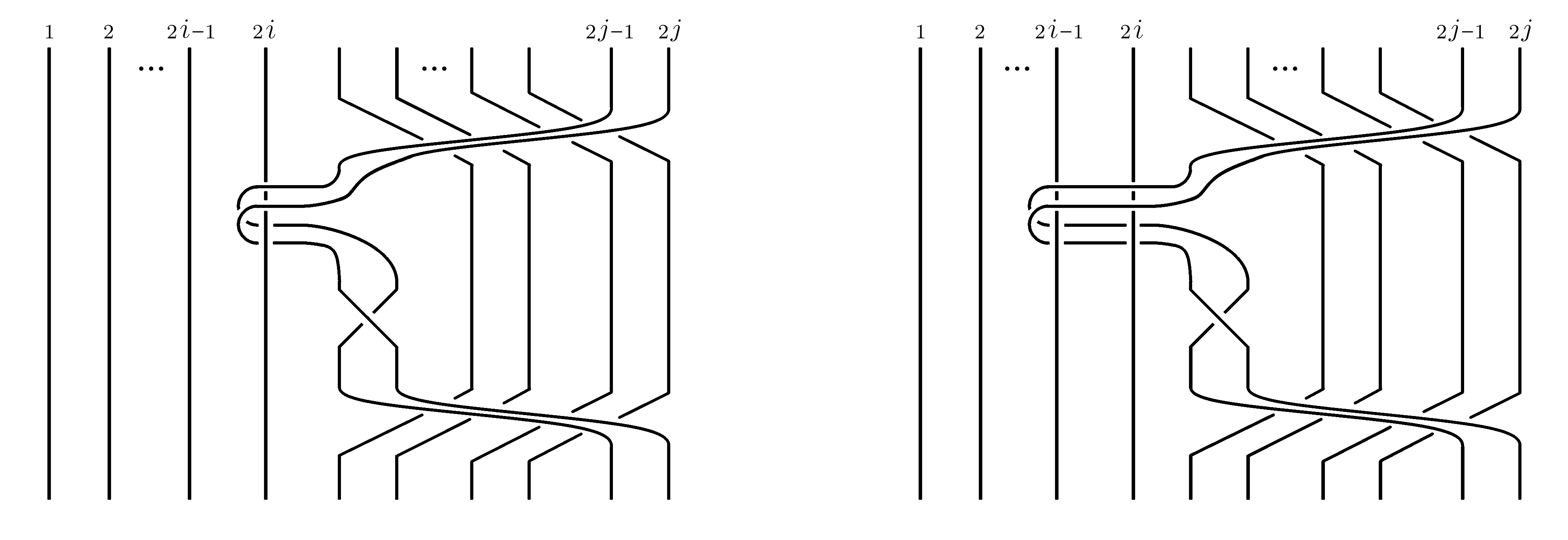}
    \caption{The \(\rho_{ij}\) and \(\omega_{ij}\) generators in the case of \(i<j \)}
    \label{fig:gen_hilden_rho_omega}
\end{figure}

A visual representation of the  minimal set of generators  is depicted in Figure~\ref{fig:gen_hilden}. 

\begin{figure}[H]
    \centering
    \includegraphics[width = .55\textwidth]{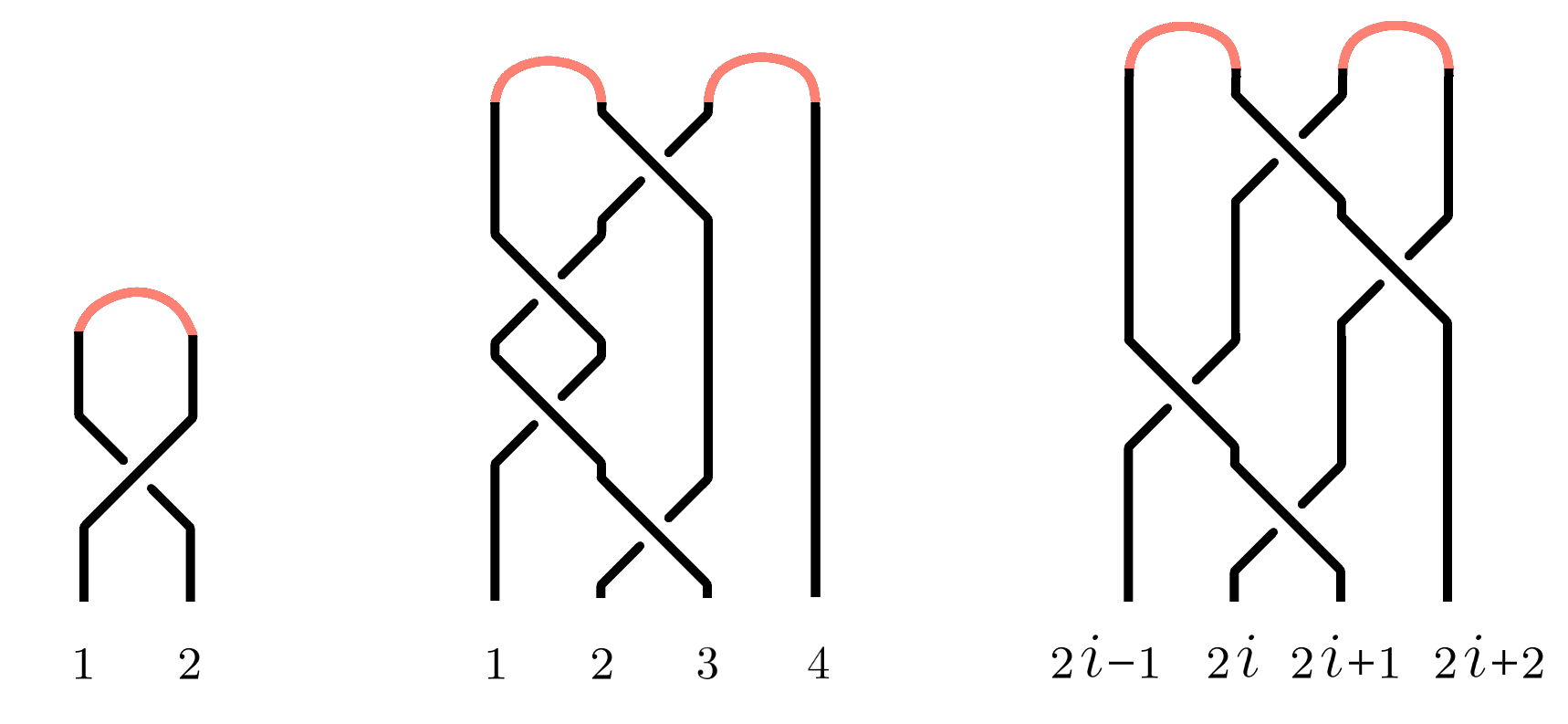}
    \caption{The three types of the  minimal set of generators for \(K_{2m}: \sigma_1, \lambda_1, \mu_i\). }
    \label{fig:gen_hilden}
\end{figure}

\subsection{Plat equivalence}

In \cite{birman1976stable} Birman proves that every link can be represented as the plat closure of a braid in some \(\mathcal{B}_{2n}\). She also proves a result concerning equivalence between plat closures of braids: 

\begin{theorem}\cite{birman1976stable} \label{Thm:birman_original}
Let \(L_1, L_2\) be two links, and let \(B_1 \in \mathcal{B}_{2n_1}, B_2 \in \mathcal{B}_{2n_2}\) be two braids such that their plat closures define the same link types as \(L_1, L_2\) respectively. Then \(L_1\) is equivalent to \(L_2\) if and only if there exists an integer \(t \geq \max(n_1, n_2)\) such that, for each \(m \geq t\) the elements: 
\[B_i' = B_i \sigma_{2n_i} \sigma_{2n_i+2} \dots \sigma_{2m-2} \in \mathcal{B}_{2m}, \qquad i = 1, 2\]
are in the same double coset of \(\mathcal{B}_{2m}\) modulo the Hilden subgroup \(K_{2m}\). 
\end{theorem}
\noindent  Fig.~\ref{fig:birman_th} illustrates abstractly the move in Theorem~\ref{Thm:birman_original}.  It can be seen using Reidemeister I moves that these two braids, closed in the plat way, give rise to the same link isotopy class. 
\begin{figure}[H]
    \centering
    \includegraphics[width = .4\textwidth]{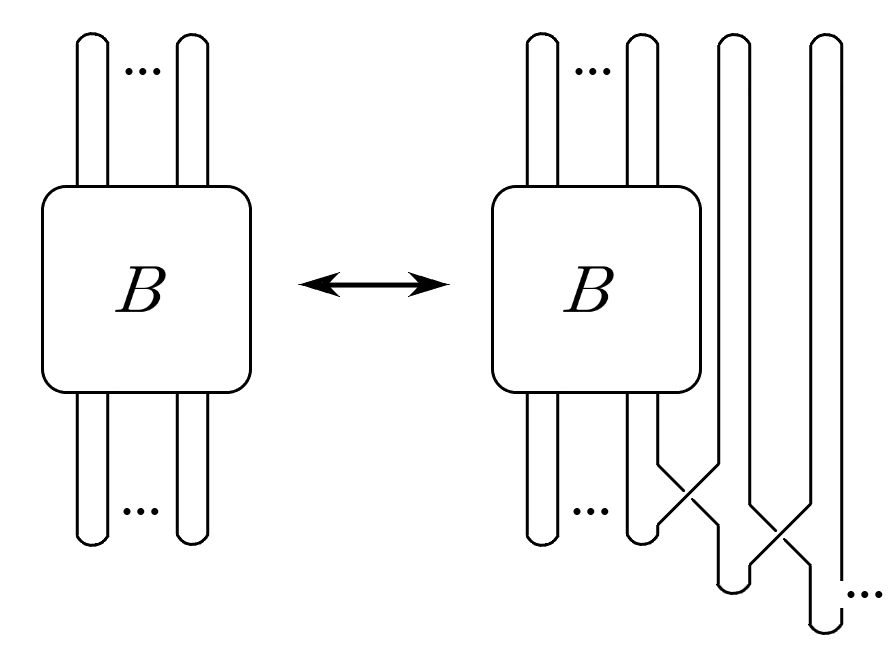}
    \caption{A braid \(B \in \mathcal{B}_{2n}\) and the braid \(B \sigma_{2n_i} \sigma_{2n_i+2} \dots \sigma_{2m-2} \in \mathcal{B}_{2m}\).}
    \label{fig:birman_th}
\end{figure}

\begin{figure}[H]
    \centering
    \includegraphics[width = .7\textwidth]{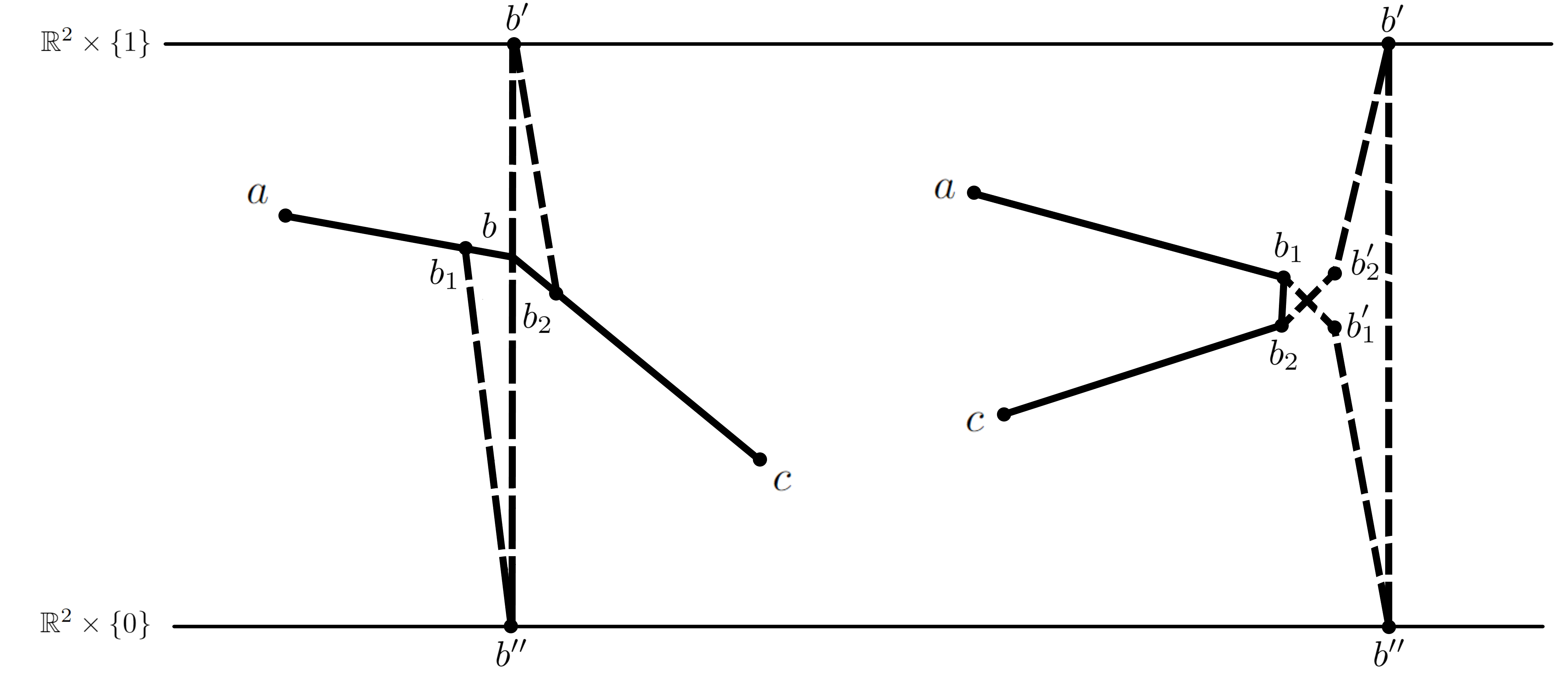}
    \caption{Two equivalent forms of the stabilization move.}
    \label{fig:classic_stabilization}
\end{figure}

The equivalence move in Theorem~\ref{Thm:birman_original}, as illustrated in Fig.~\ref{fig:birman_th}, is a multiple application of the {\it stabilization move} for plats (for a proper definition see Section~\ref{Equivalence}, Definition~\ref{def:mixedstabilization}). The stabilization move is depicted in Fig.~\ref{fig:classic_stabilization} in two equivalent forms up to plat equivalence: the right-hand form is created from the one on the left by performing a Reidemeister II move at the top part of the vertical arc and then suppressing one of the crossings by the equivalence; and vice-versa. In the Figure we, deliberately, have not indicated an over-arc for the crossing, as both types apply.

\smallbreak 
A crucial observation is now due.

\begin{remark} \rm 
Both instances of the stabilization move utilized in the equivalence relation between  plats are identical to the standardly closed versions of the \(L\)-moves (recall  Fig.~\ref{fig:L_move} in Section \ref{Section:standard_braiding_equivalence}), forgetting orientation. 
\end{remark}

Throughout, we will usually employ capital letters for braids closed in plat form.

\section{Braids and plats in handlebodies} \label{Handlebody}

In this section we first discuss links and braids in a handlebody and recall their representations in \(\mathbb{R}^3\), the mixed links and mixed braids  \cite{lambropoulou1997markov,haring2002knot}. We also recall from \cite{haring2002knot} the standard closure of  mixed braids, as adapted to handlebodies, and the analogues of braiding and braid equivalence for the standard closure. We proceed with defining the notion of plat closure for mixed braids and provide an algorithm for turning a link in a handlebody into a plat, directly analogous to that of \cite{birman1976stable} for \(\mathbb{R}^3\), and also to  \cite{cattabriga2018markov} for thickened surfaces and for closed, connected and orientable 3-manifolds. We proceed with defining the appropriate Hilden braid subgroup for the mixed braid group and formulate and prove the plat closure mixed braid equivalence in the setting of handlebodies. 



\subsection{Handlebody representation in $\mathbb{R}^3$} 

Let \(H_g\) be the abstract handlebody of genus \(g\). \(H_g\) can be constructed from a 3-ball with $g$ 1-handles attached (see left part of Fig.~\ref{fig:standard_handles}). Note that, for $g=0$ then \(H_0\) is just the 3-ball, while for $g=1$ then \(H_1\) is the solid torus.  Equivalently, \(H_g\)  can be constructed as the  product space of the $g$-punctured disc, $D_g$, cross the interval (middle illustration of Fig.~\ref{fig:standard_handles}): 
 \(H_g = D_g \times [0,1].\)

\begin{figure}[H]
    \centering
    \includegraphics[width = \textwidth]{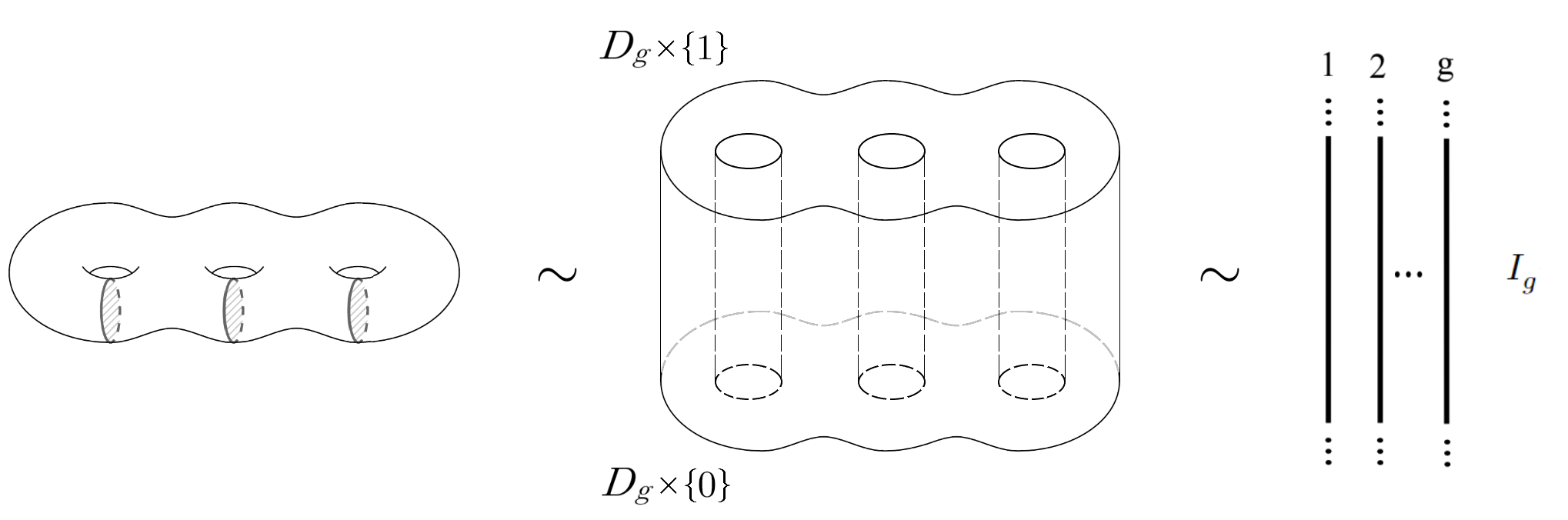}
    \caption{Abstract handlebody of genus 3 and its representation in  \(\mathbb{R}^3\). }
    \label{fig:standard_handles}
\end{figure}

Hence \(H_g\),  as embedded in  \(\mathbb{R}^3\), may be viewed as the complement of the identity braid \(I_g\) on \(g\) indefinitely extended strands  (see right illustration of Fig.~\ref{fig:standard_handles}). For viewing \(H_g\) in \(S^3\) we consider the "compactification", namely all strands of \(I_g\)  meeting at the point at infinity.  (Then a thickening of the resulting graph is the complementary handlebody of \(H_g\) in \(S^3\)).

\subsection{Links in handlebodies and mixed links}

Let now \(L\) be a link in \(H_g\). Then, using the above representation. we can consider \(L\) to lie in the interior of \(H_g = D_g \times [0,1]\), except perhaps for finitely many points lying in the boundary components \(D_g \times \{0,1\}\). Further, fixing \(I_g\) pointwise, \(L\) may be unambiguously represented by the mixed tangle (or {\it mixed link}) \(I_g  \cup L\), see left illustration of Fig.~\ref{fig:mixed_tangle}. $I_g$ forms the so-called {\it fixed part} of the mixed link, while $L$ forms the so-called {\it moving part} of the mixed link. If we remove $I_g$ from \(I_g  \cup L\) we are left with a link $L$ in $\mathbb{R}^3$, which we shall call a {\it template} for the initial link $L$ in $H_g$. Note that, for a given template $L$ in $\mathbb{R}^3$ there are infinitely many links in $H_g$ having $L$  as the moving part and $I_g$ as the fixed part, that is, there are infinitely many links in $H_g$ with the given template. Further, two (polygonal) links in $H_g$ are called {\it isotopic} if there is a finite sequence of the well-known $\Delta$-moves taking one to the other. An analogous definition carries through to the corresponding mixed links, see  \cite{lambropoulou1997markov,haring2002knot}.

\begin{figure}[H]
    \centering
    \includegraphics[width = 1\textwidth]{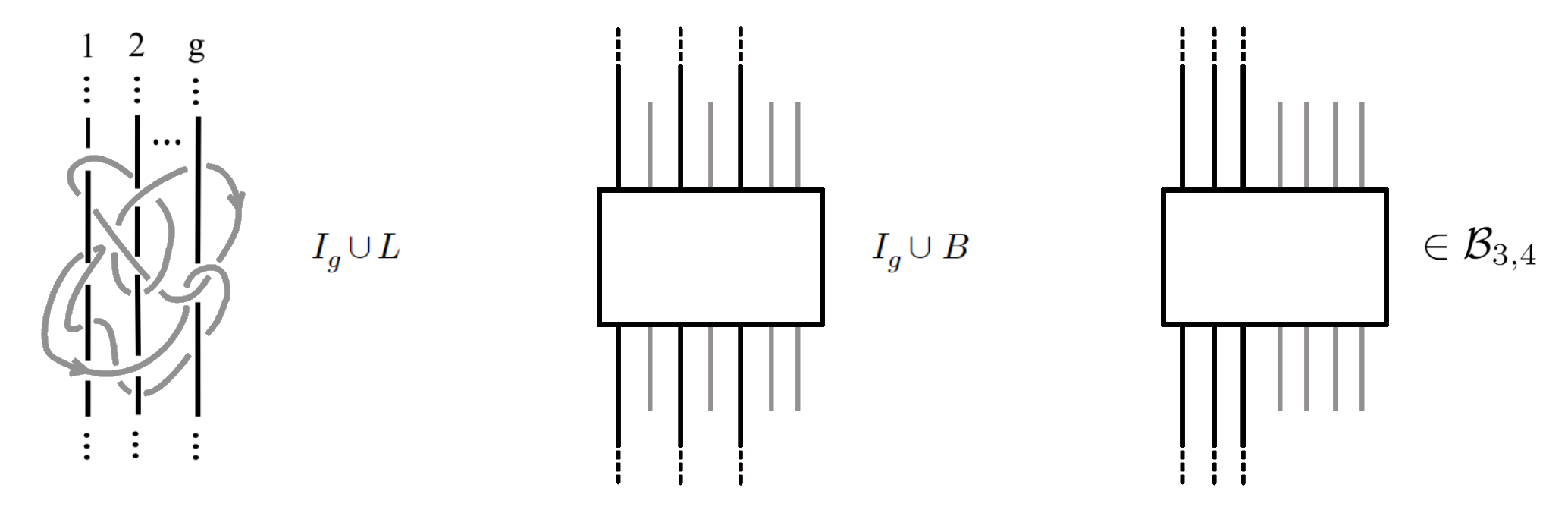}
    \caption{left: a mixed link, middle: an abstract geometric mixed braid, right: an abstract algebraic mixed braid.}
    \label{fig:mixed_tangle}
\end{figure}

\subsection{Braids in handlebodies,  mixed braids and the mixed braid group}\label{Mixed_braids}

\begin{definition} \rm 
A {\it braid  $\beta$ in $H_g$ on $n$ strands} is defined as an embedding \(\phi\) of \(n\) distinct copies of \(I = [0,1]\), \(I_j = [0,1], j = 1, \dots, n\) in \(H_g\), such that each \(\phi_j := \phi |_{I_j}\) is monotonous, and  the top $n$ endpoints of the braid $\beta$ lie in $D_g \times \{0\}$ and the corresponding bottom \(n\) endpoints of $\beta$ lie in $D_g \times \{1\}$, in the realization of $H_g $ as $ D_g \times [0,1]$. 
\end{definition}

Following \cite{lambropoulou1997markov,haring2002knot}, the braid  $\beta$  in $H_g$  is represented uniquely in $S^3$ by a {\it geometric mixed braid} on $n$ moving strands, denoted $I_g \cup \beta$, which is an element of the Artin braid group \(\mathcal{B}_{g+n}\)  and which contains the identity braid $I_g$ on \(g\) strands as a fixed subbraid (see abstraction in middle illustration of Fig.~\ref{fig:mixed_tangle}). 

By removing $I_g$  we are left with the {\it moving subbraid} $\beta$ on $n$ strands. We shall call this braid $\beta$ in $\mathbb{R}^3$ a {\it template} for the initial braid $\beta$ in $H_g$. Note that, for a given template $\beta$ in $\mathbb{R}^3$ there are infinitely many geometric mixed braids having $\beta$ as the moving subbraid and $I_g$ as the fixed subbraid and, consequently, infinitely many braids in $H_g$ with the given template. 

A special case of geometric mixed braid on $g+n$ strands  is one for which, by removing the last \(n\) strands  we are left with the identity braid  $I_g$.  Such a mixed braid shall be called {\it algebraic mixed braid} (notion introduced in  \cite{lambropoulou2000braid}). See abstraction in rightmost illustration of Fig.~\ref{fig:mixed_tangle}). We shall use the same notation for  algebraic mixed braids as for the generality of  geometric mixed braids.

It is established in \cite{lambropoulou2000braid} that the set of all algebraic mixed braids on $g$ fixed strands and $n$ moving strands form a group, the {\it  mixed braid group},  \(\mathcal{B}_{g,n}\). In other words, \(\mathcal{B}_{g,n}\) is the subgroup of all elements in  \(\mathcal{B}_{g+n}\) for which the first $g$ strands form the identity braid  $I_g$. Of course, for \(g = 0\), we obtain the standard definition of the Artin braid group in \(\mathbb{R}^3\), \(\mathcal{B}_n\), while   for \(g = 1\), we obtain the Artin braid group of type $B$, $\mathcal{B}_{1,n}$, which is the algebraic counterpart of knot theory in the solid torus and in the lens spaces (see \cite{lambropoulou1994torus}).  It is further established in \cite{haring2002knot} that the groups \(\mathcal{B}_{g,n}\), $n\in \mathbb{N}$, comprise the algebraic counterpart for knots and links in \(H_g\).  

In \cite{lambropoulou2000braid}, it is proved that \(\mathcal{B}_{g,n}\) has a presentation with generators: the elementary crossings among moving strands  $\sigma_1, \dots, \sigma_{n-1}$ (see Fig.~\ref{fig:generat_L}) and  the {\it loop generators } $\alpha_1, \dots, \alpha_g$, where \(\alpha_i\) represents a right-handed looping of the first moving strand around the $i$-th fixed strand, as depicted  in Fig.~\ref{fig:generat_L}, where also the inverse of a loop generator is depicted. Further, the generators satisfy the relations: 
\begin{equation} 
\begin{array}{rcll}
\sigma_k \sigma_j &=& \sigma_j \sigma_k & |k-j|>1, \\ 
\sigma_k \sigma_{k+1} \sigma_k &=& \sigma_{k+1} \sigma_k \sigma_{k+1} & 1\leq k\leq n-1,  \\
\alpha_r \sigma_k &=& \sigma_k \alpha_r  & k\geq 2, \  1\leq r \leq g, \\
\alpha_r \sigma_1 \alpha_r \sigma_1 &=& \sigma_1 \alpha_r \sigma_1 \alpha_r & 1 \leq r \leq g, \\
\alpha_r (\sigma_1 \alpha_s \sigma_1^{-1}) &=& (\sigma_1 \alpha_s \sigma_1^{-1}) \alpha_r & s < r.
\end{array}
\end{equation}

\begin{figure}[H]
    \centering
    \includegraphics[width = .8\textwidth]{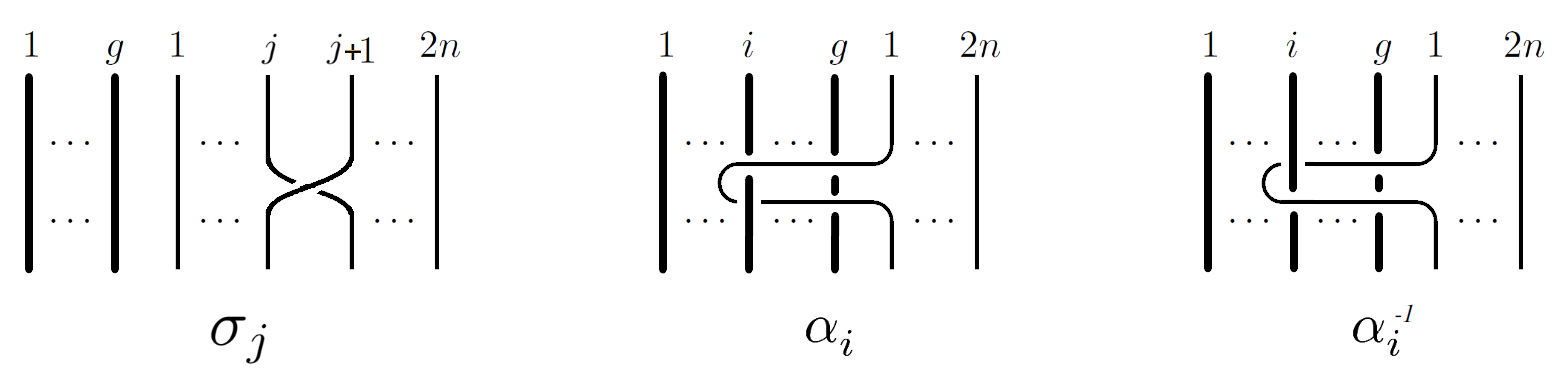}
    \caption{The generators of \(\mathcal{B}_{g,n}\).}
    \label{fig:generat_L}
\end{figure}

From the above, an algebraic mixed braid $I_g \cup \beta$ on $n$ moving strands represents uniquely in $\mathbb{R}^3$  a braid  $\beta$ on $n$  strands in \(H_g = D_g \times [0,1]\), such that its endpoints lie all  within a disc on the right-hand side of $D_g$. This leads to the following:

\begin{definition}\label{def:algebraic_braid}
 A braid  $\beta$ on $n$  strands in \(H_g = D_g \times [0,1]\), such that its  endpoints lie all  within a disc on the right-hand side of $D_g$ shall be called {\it algebraic braid in \(H_g\)}.  
\end{definition}

\begin{remark}\label{rmk:braid_surface_disk} \rm 
Fixing a set of $n$ points \( \mathcal{C} = \{C_1,\ldots,C_n\} \) in $D_g$ we observe that the set of all braid on $n$ strands in $H_g$, with endpoints in  \( \mathcal{C}\) forms a group by concatenation along the surface $D_g$. Clearly, groups related to different sets of endpoints are isomorphic. So, without loss of generality, one may assume that the set \( \mathcal{C}\) lies in a disc $A$ on the right-hand side of $D_g$, as in Fig.~\ref{fig:handlebody_disk}.  By definition, these groups for different sets of endpoints coincide with the {\it surface braid group of $D_g$ on $n$ strands}, which we shall denote $\mathcal{B}_{D_g,n}$ (the surface braid groups were introduced in \cite{fadell1962configuration}, see also \cite{gonzalez2001new, bellingeri2004presentations}). Furthermore, by the well-defined correspondence between mixed braids and braids in $H_g$, and of their operations, the surface braid group  $\mathcal{B}_{D_g,n}$ is easily shown to be isomorphic to the mixed braid group \(\mathcal{B}_{g,n}\).
\end{remark}

\begin{figure}
    \centering
    \includegraphics[width = .5\textwidth]{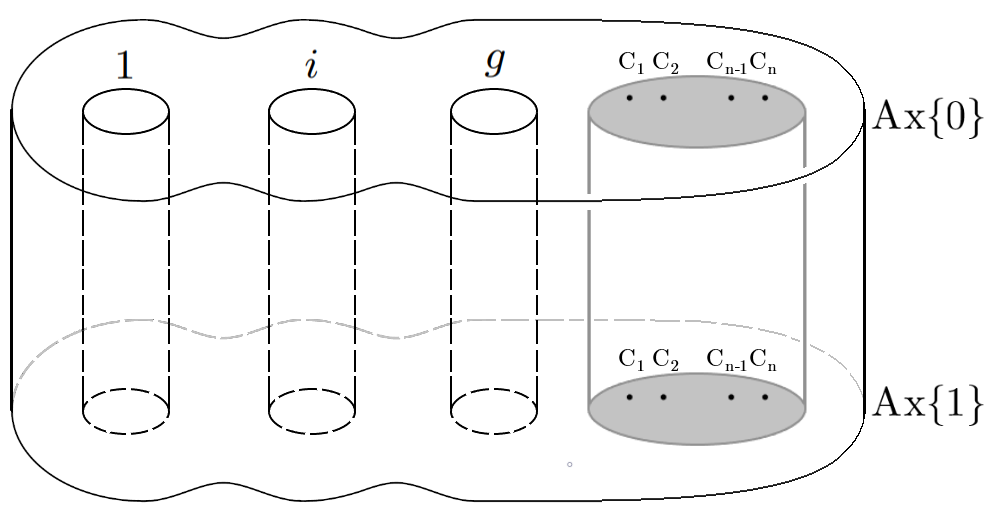}
    \caption{The set of points \(C_1, C_2, \dots, C_n\) on the two levels \(D_g \times \{0\}\) and \(D_g \times \{1\}\).}
    \label{fig:handlebody_disk}
\end{figure}

\section{Representing links via braids in handlebodies }\label{sec:plat_closure_mixed}

\subsection{The standard closure for mixed braids, the mixed braiding and the mixed braid equivalence }\label{}

Let $I_g\cup \beta$ be a geometric mixed braid, representing uniquely the braid \(\beta\) in \(H_g\). We recall from \cite{haring2002knot} that, in order to define well the (standard) closure of $I_g\cup \beta$, we first label each pair of corresponding endpoints of \(\beta\) with `o' (for over) or `u' (for under). Then, the {\it standard closure} of \(I_g \cup \beta \) is an oriented mixed link obtained by joining with simple unlinked arcs each pair of corresponding endpoints of the moving subbraid \(\beta\), according to its label: the simple arc will be lying entirely over or entirely under the rest of the mixed braid, if its label is `o'  or `u' respectively. See Fig.~\ref{fig:mixed_tangle_2}(a) for an abstract illustration. 

Note that, the definition of standard closure for mixed braids induces the definition of {\it standard closure for braids in \(H_g = D_g \times [0,1]\)}.  

\begin{figure}[H]
    \centering
    \includegraphics[width = .65\textwidth]{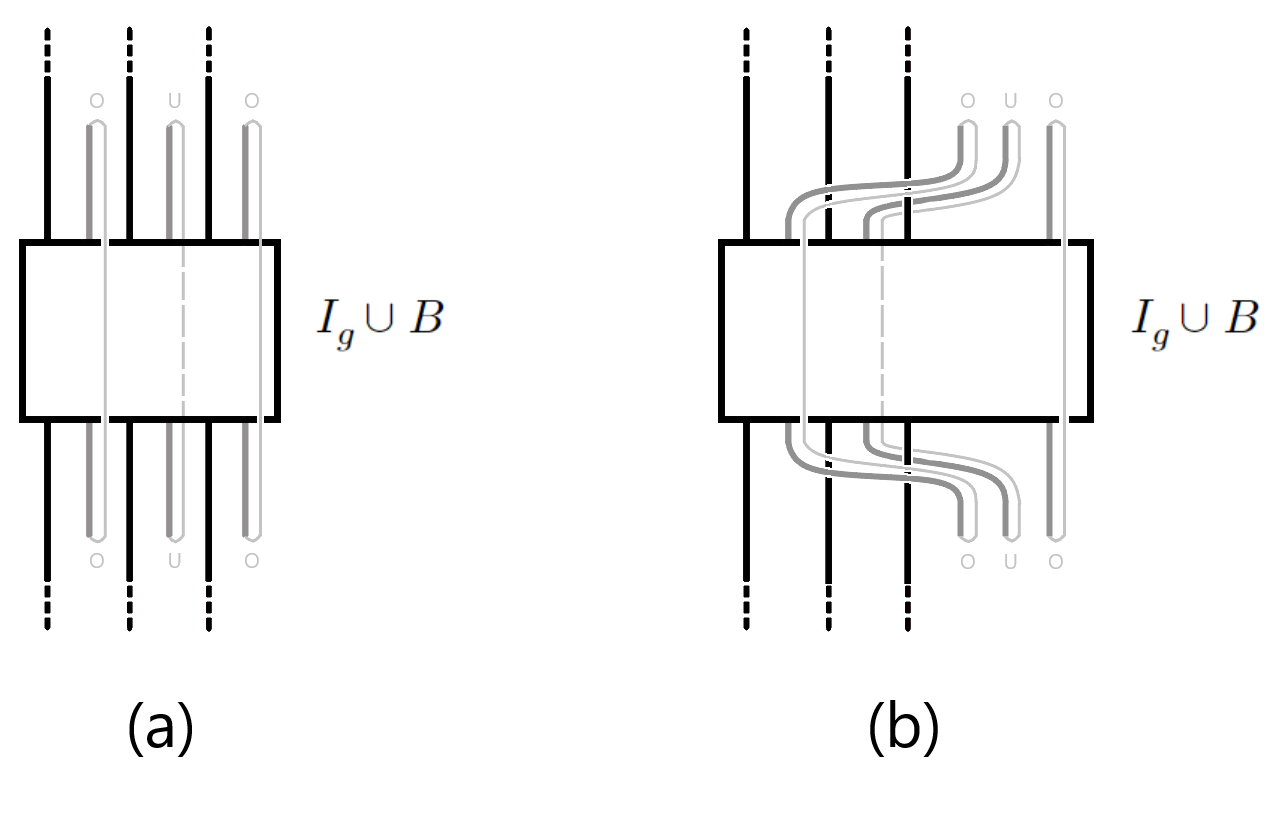}
    \caption{(a) the standard closure of a geometric mixed braid; 
    (b) a parted mixed braid and its homologous closure. }
    \label{fig:mixed_tangle_2}
\end{figure}

As pointed out in \cite{haring2002knot},  different choices of labels for a geometric mixed braid will, in general, result in non-isotopic mixed links. Yet, this is not the case if one restricts to algebraic mixed braids, since a closing arc can freely move through isotopy from the front to the back of the mixed braid (and vice versa), without intersecting any fixed strand.   

The inverse of the closing operation is also true \cite{haring2002knot}: 
any oriented mixed link (representing uniquely an oriented link in $H_g$), can be braided to a geometric mixed braid with isotopic standard closure. This is an analogue of the classical Alexander theorem for knots and links in $H_g$.  Furthermore,  the parting operation enables a geometric mixed braid to be represented  by an algebraic mixed braid with isotopic standard closure. Fig.~\ref{fig:mixed_tangle_2}(b) depicts a parted mixed braid and its homologous closure.  Hence, {\it any oriented mixed link can be braided to an algebraic mixed braid with isotopic standard closure.} This is the algebraic analogue of the  Alexander theorem for knots and links in $H_g$. 

Moreover, mixed link isotopy for oriented links in $H_g$ is translated in \cite{haring2002knot} into  mixed braid equivalence: first as $L$-equivalence among geometric mixed braids (an analogue of the one-move Markov theorem), and then as an algebraic mixed braid equivalence in $\bigcup_n{\mathcal{B}_{g,n}}$ (an analogue of the classical Markov theorem). For details the reader may consult \cite{haring2002knot}.

\subsection{The plat closure for braids in  $H_g$ and for mixed braids}\label{Main_result}

\begin{definition}
Let $I_g\cup B$ be a geometric mixed braid with an even number of moving strands, say $2n$. For defining the plat closure of $I_g\cup B$ we first number the set of top endpoints of $B$, from left to right, with numbers $1, 1^\prime, 2, 2^\prime, \ldots, n, n^\prime$. To each pair $(i, i^\prime)$ a label `o' (for over) or `u' (for under) is assigned. Then, we use a simple arc, say $\gamma_i$, as joining arc for the pair of endpoints $(i, i^\prime)$, so that $\gamma_i$ runs entirely over or entirely under any fixed strand of the mixed diagram, according to the label of $(i, i^\prime)$.  We repeat the above procedure, independently, for the set of corresponding bottom endpoints. `Independently' in the sense that the labels of the closing arcs, say $\delta_i$, for joining the pairs of bottom endpoints are not related to the labels of the closing arcs $\gamma_i$ for joining the pairs of top endpoints. View Fig.~\ref{fig:example_plats} for an example.  The resulting mixed link is a {\it plat closure} for $I_g\cup B$, called {\it mixed plat}, and shall be denoted $I_g\cup \overline{B}$. 

In the special case where  $I_g\cup B$ is an algebraic mixed braid the resulting plat shall be called {\it algebraic mixed plat}.
\end{definition}

\begin{remark} \rm 
As for the case of standard closure, the definition of plat closure for mixed braids induces the definition of {\it plat closure for braids in \(H_g = D_g \times [0,1]\)}. Only, here we also add small local isotopies of the closing arcs of a braid $B$, so that the resulting plat, denoted $\overline{B}$, lies in the interior of $H_g$, except perhaps finitely many points. 
\end{remark}

\begin{remark}\label{labels_plat} \rm
Different labels for closing arcs may result in non-isotopic plats. An example is illustrated in Fig.~\ref{fig:example_plats}.  Yet, note that  the label of a joining arc is irrelevant if no fixed strands of $I_g$ lie between the two endpoints. It follows, in particular, that if $I_g\cup B$ is an algebraic mixed braid, then no labels are required for forming the plat closure, as no choices are involved in this case and the plat closure is determined uniquely.   
\end{remark}

\begin{figure}[H]
    \centering
    \includegraphics[width = .6\textwidth]{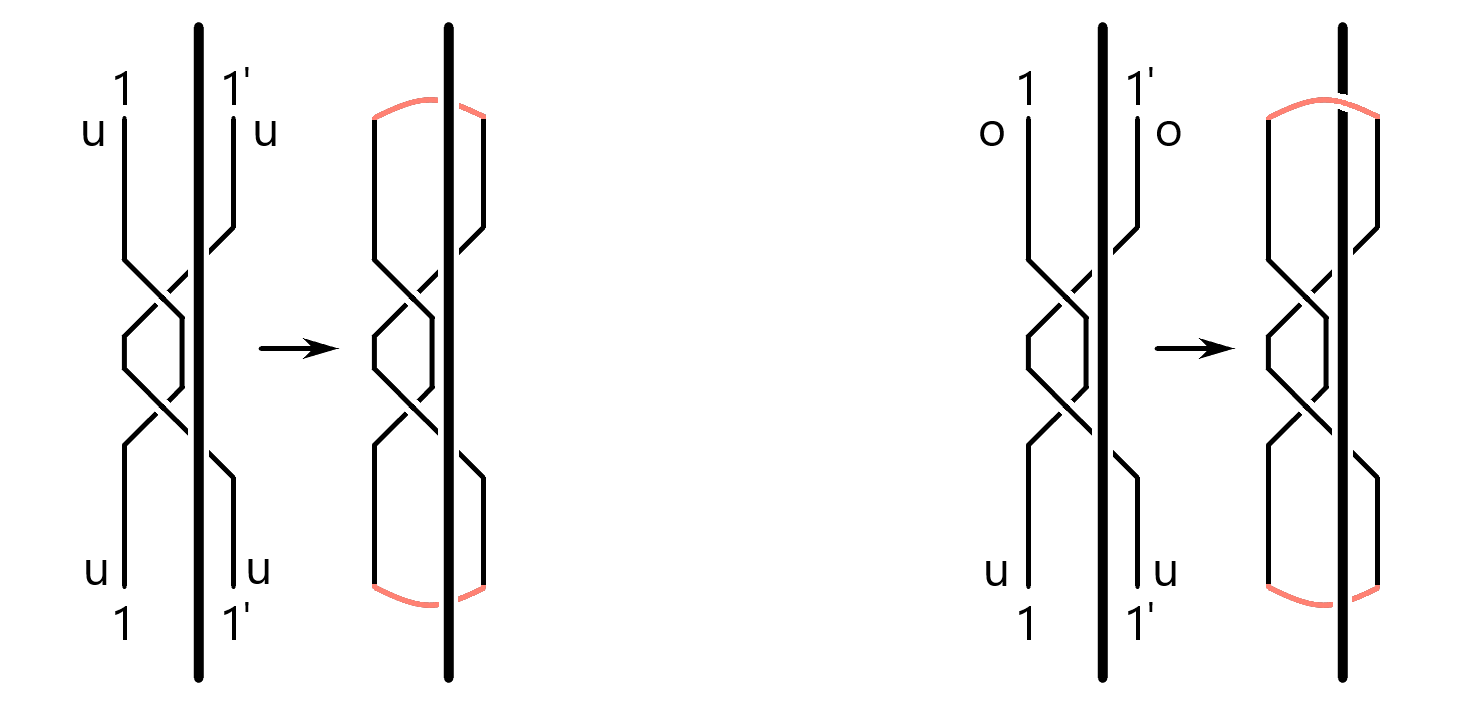}
    \caption{Two different examples of plat closure of a mixed braid. The two resulting links are not isotopic. }
    \label{fig:example_plats}
\end{figure}

\subsection{Representing mixed links as plat closures of mixed braids}\label{Main_theorem}

Let  \(L\) be a link in $H_g$ represented by the mixed link \(I_g\cup L\). We want to show that the mixed link \(I_g\cup L\) may be isotoped to the (labelled) plat closure of a geometric mixed braid, providing thus an analogue of the Birman braiding theorem for knots and links in $H_g$. Namely, we have the following, which is also an analogue of the  braiding theorem in $H_g$, but for plats.

\begin{theorem}\label{teo_braiding_plat_hand} 
Let  \(L\) be a link in  the genus \(g\) handlebody  $H_g$, represented by the mixed link \(I_g\cup L\) in $\mathbb{R}^3$. Then \(L\)  may be braided to a   braid in $H_g$ with even number of  strands, whose  plat closure is isotopic to \( L\) in  $H_g$. Equivalently, \(I_g\cup L\)  may be braided to a geometric mixed braid with even number of moving strands, whose  plat closure is isotopic to \(I_g\cup L\). 
\end{theorem}

\begin{proof} 
We shall adapt the braiding algorithms of \cite{birman1976stable, cattabriga2018markov} to our setting. Indeed, let \(\overrightarrow{D_g}\) be the unit vector orthogonal to \(D_g\).
It is always possible (using local isotopies, if needed) to consider \(L \subset D_g \times [0.25, 0.75]\). 


Consider a polygonal representation of \(L\), with set of vertices \(\{b_0, b_1, \dots, b_n\}\) and with no edges parallel to \(D_g \times \{0\}\) (which is always possible by a general position argument). Following \cite[Lemma 2]{birman1976stable}, we have that for any 3 consecutive vertices \(b_{j-1}, b_j, b_{j+1}\) of \(L\) the line \(l_j\) orthogonal to \(D_g\) through \(b_j\) admits a neighbourhood \(N_j\) which meets the link only in the edges \([b_{j-1}, b_j]\) and \([b_j, b_{j+i}]\) (see Fig.~\ref{fig:braiding_2}). We will say that \(L\) is in standard position. 

\begin{figure}[H]
    \centering
    \includegraphics[width = .7\textwidth]{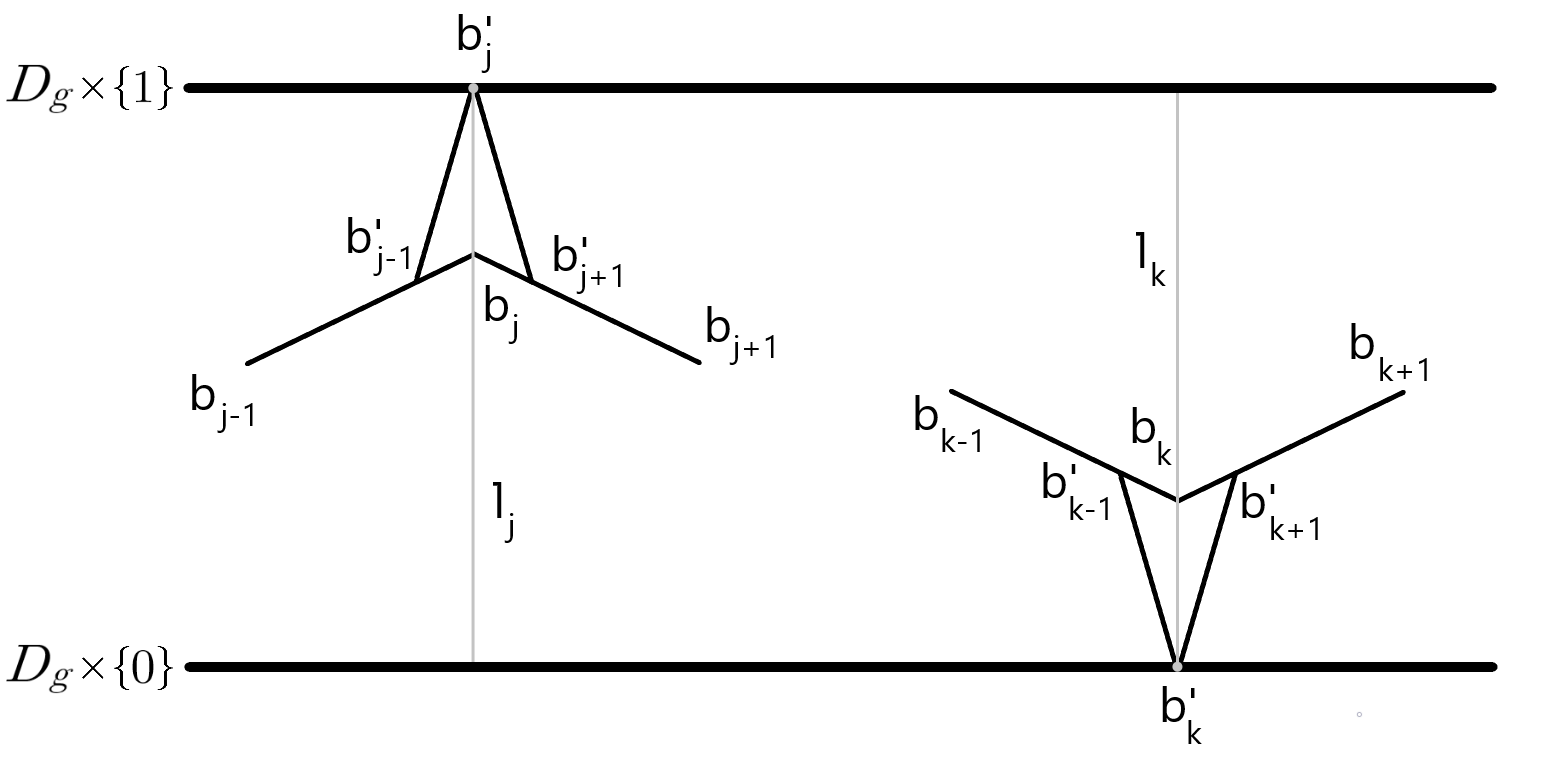}
    \caption{The braiding process. }
    \label{fig:braiding_2}
\end{figure}

Then, starting from \(b_0\), for every three points \(b_{j-1}, b_j, b_{j+1}\) with 
\[\left(\overrightarrow{[b_{j-1}, b_j]} \cdot \overrightarrow{D_g}\right)\cdot \left(\overrightarrow{[b_j, b_{j+1}]} \cdot \overrightarrow{D_g}\right) < 0\] 
(meaning that \(b_j\) is either a local maximum or a local minimum) consider the intersections \([b_{j-1}, b_j] \cap N_j = b'_{j-1}\), \([b_j, b_{j+1}] \cap N_j = b'_{j+1}\). If \(\overrightarrow{[b_{j-1}, b_j]} \cdot \overrightarrow{D_g} >0\) (resp. \(<0\)) let \(b'_j = l_j \cap D_g \times \{1\}\) (resp \(b'_j = l_j \cap D_g \times \{0\}\)). If we replace the edge sequence \(b_{j-1}, b_j, b_{j+1}\) with \(b_{j-1}, b'_{j-1}, b'_j, b'_{j+1}, b_{j+1}\) the isotopy class of the link does not change, while now the link intersects once \(D_g \times \{0, 1\}\). 

Repeating this operation for every three points satisfying the above condition, we will obtain an isotopic link meeting the discs \(D_g \times \{0\}\) and \(D_g \times \{1\}\) in exactly, say, \(n\) points, and every other disc  \(D_g \times \{t\}, t \in (0,1)\), in exactly \(2n\) points. It follows from Artin's definition of a braid \cite{artin1947theory} that such a link determines a well-defined geometric mixed braid \(I_g \cup B \) on $2n$ moving strands. In turn, this mixed braid represents the link  \(L\)  via plat closure in \(H_g = D_g \times [0,1]\). 
\end{proof}

Furthermore, we prove that any mixed  link (resp. link in \(H_g\)) may be braided to (resp. represented by)  an {\it algebraic  mixed braid} (recall Definition \ref{def:algebraic_braid}) via plat closure. For this we first recall from \cite{birman1976stable} and adapt the definition of the spike move.

\begin{definition} \label{def:mixedspike}
Let \(\overline{B}\) be a link  in \(H_g\) in plat form, represented in \(\mathbb{R}^3\) by the mixed plat $I_g \cup \overline{B}$. Let also \(b_i, b_{i+1}, b_{i+2}\) be three consecutive points of \(\overline{B}\), such that \(b_{i+1}\) is a boundary point (i.e. \(\in D_g \times \{0,1\}\)) and such that the triangle \([b_i, b_{i+1}, b_{i+2}]\) is entirely contained in \(D_g \times [0,1]\) and intersects \(\overline{B}\) only on its sides \([b_i, b_{i+1}]\) and \([b_{i+1}, b_{i+2}]\) (see Fig.~\ref{fig:spike_move}). Edges like \([b_i, b_{i+1}], [b_{i+1}, b_{i+2}]\) are called \textit{boundary edges}. Let now \(b'_{i+1}\) be a new boundary point on the same boundary component as \(b_{i+1}\) and let \(b'_i, b'_{i+2}\) be interior points  satisfying the conditions:
\begin{itemize}
    \item \([b'_{i}, b'_{i+1}, b'_{i+2}] \cap L = \emptyset\) and  \([b'_{i}, b'_{i+1}, b'_{i+2}]\) lies entirely in \(D_g \times [0,1]\);
    \item for the region \([b_i, b'_i, b'_{i+2}, b_{i+2}]\) bounded by the points \(b_i, b'_i, b'_{i+2}, b_{i+2}\) we have: \\ \([b_i, b'_i, b'_{i+2}, b_{i+2}] \subset D_g \times [0,1] \ \ \& \ \ [b_i, b'_i, b'_{i+2}, b_{i+2}] \cap L = \{b_i, b_{i+2}\}\). 
\end{itemize}
A \emph{spike move} for the plat $\overline{B}$ consists in replacing the edges defined by the sequence \(b_{i}, b_{i+1}, b_{i+2}\) by the edges defined by the sequence \(b_{i}, b'_{i}, b'_{i+1}, b'_{i+2}, b_{i+2}\). Then the isotopy class of the plat \(\overline{B}\) does not change and the spike move preserves the plat form. 

A {\it spike move for the mixed plat} $I_g \cup \overline{B}$ is defined analogously, except that two of the conditions above are replaced by the conditions: 
$$
[b'_{i}, b'_{i+1}, b'_{i+2}] \cap (I_g \cup \overline{B}) = \emptyset \ 
\mbox{and} \ 
[b_i, b'_i, b'_{i+2}, b_{i+2}] \cap I_g = \emptyset.
$$
The conditions imply, in particular, that if a fixed strand overlies (resp. underlies) the triangle \([b_i, b_{i+1}, b_{i+2}]\) then the fixed strand also overlies (resp. underlies) the replacement of the spike move. Then the isotopy class of the mixed plat $I_g \cup \overline{B}$  does not change and the spike move preserves the mixed plat form. 
\end{definition}

\begin{figure}[H]
    \centering
    \includegraphics[width = .5\textwidth]{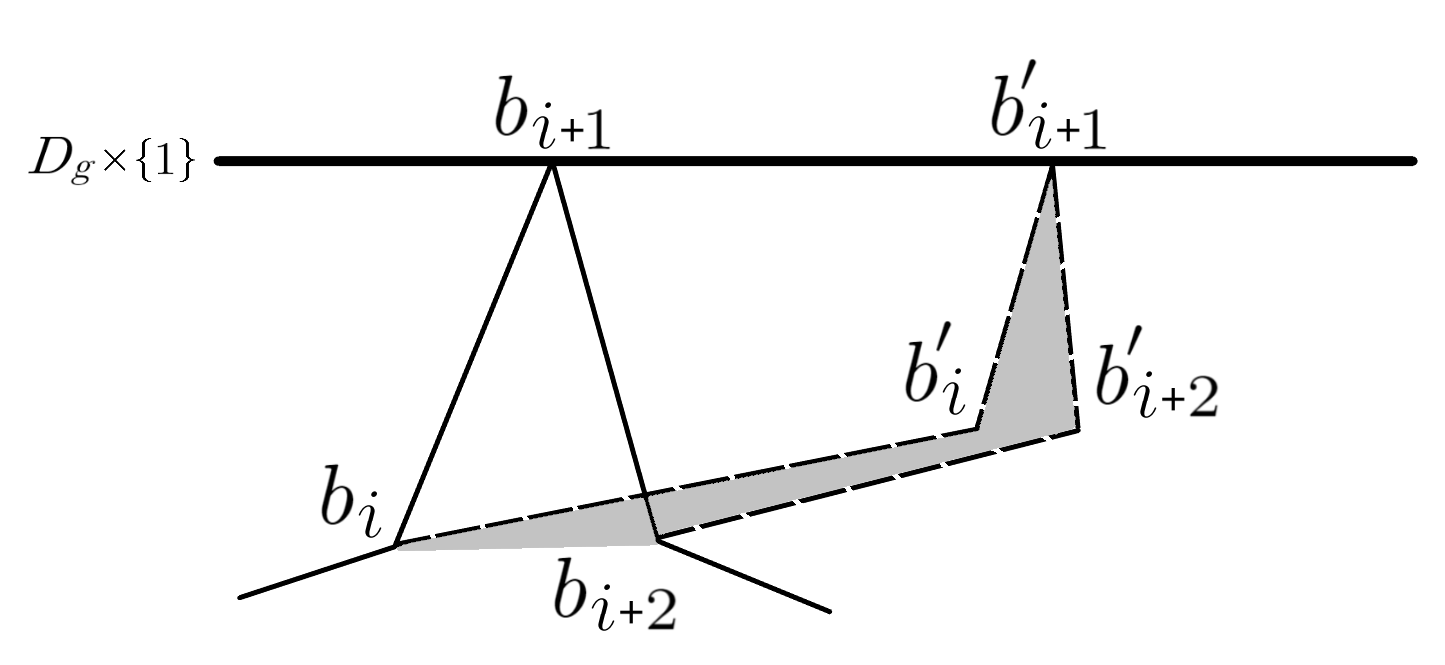}
    \caption{A spike move. }
    \label{fig:spike_move}
\end{figure}

We can now sharpen Theorem~\ref{teo_braiding_plat_hand} to representing links in $H_g$ as plat closures of algebraic mixed braids. 

\begin{theorem}\label{teo:lemma_algebraic_handleb}
Let \(L\) be a link in the genus \(g\) handlebody $H_g$, represented by the mixed link \(I_g\cup L\) in $\mathbb{R}^3$. Then \(L\) may be  braided to an algebraic  braid in $H_g$  with even number of  strands,   whose  plat closure is isotopic to \(L\) in $H_g$. 
Equivalently, \(I_g\cup L\)  may be braided to an algebraic mixed braid \(I_g\cup B \in \mathcal{B}_{g,2n}\), for some $n$, representing \(I_g\cup L\) via plat closure. 
\end{theorem}

\begin{proof}
By Theorem \ref{teo_braiding_plat_hand}, we first obtain from \(I_g \cup L\) an isotopic geometric mixed plat \(I_g \cup \overline{B}\), where \(I_g \cup  B  \in \mathcal{B}_{g+2n}\) for some $n$, such that $B$ comes with a  set of $2n$ closing labels. Note that \( I_g \cup  B \) does not necessarily belong to the mixed braid group  \(\mathcal{B}_{g,2n}\). For doing this we adapt the parting technique, as in \cite{haring2002knot} for the standard closure. 

Indeed, we perform spike moves to all boundary edges which are on the left of some strands of \(I_g\), so that we part the mixed plat, bringing all the boundary edges to the right of all strands of \(I_g\) (see an example in Fig.~\ref{fig:spike_parting}). We shall refer to this procedure as a {\it plat parting} of the geometric braid. It is important to note that a spike move, by definition, respects locally the label of a closing arc. Otherwise, the closing labels do not restrict the spike moves. In the end of the parting operation we obtain an algebraic mixed plat, \(I_g \cup \overline{B^\prime}\), such that \(I_g \cup B^\prime \in \mathcal{B}_{g,2n}\).  At this moment, the \(2n\) closing labels do not matter anymore, as no fixed strands are in between moving strands, recall Remark~\ref{labels_plat}.

Equivalently, on the level of $H_g$, using Theorem \ref{teo_braiding_plat_hand} we first obtain from  $L$ a plat in $H_g$ and then, using appropriate spike moves, we bring the  resulting plat to a new one whose closing arcs lie within a disc on the right-hand side of $D_g$.
\end{proof}

\begin{figure}[H]
    \centering
    \includegraphics[width = \textwidth]{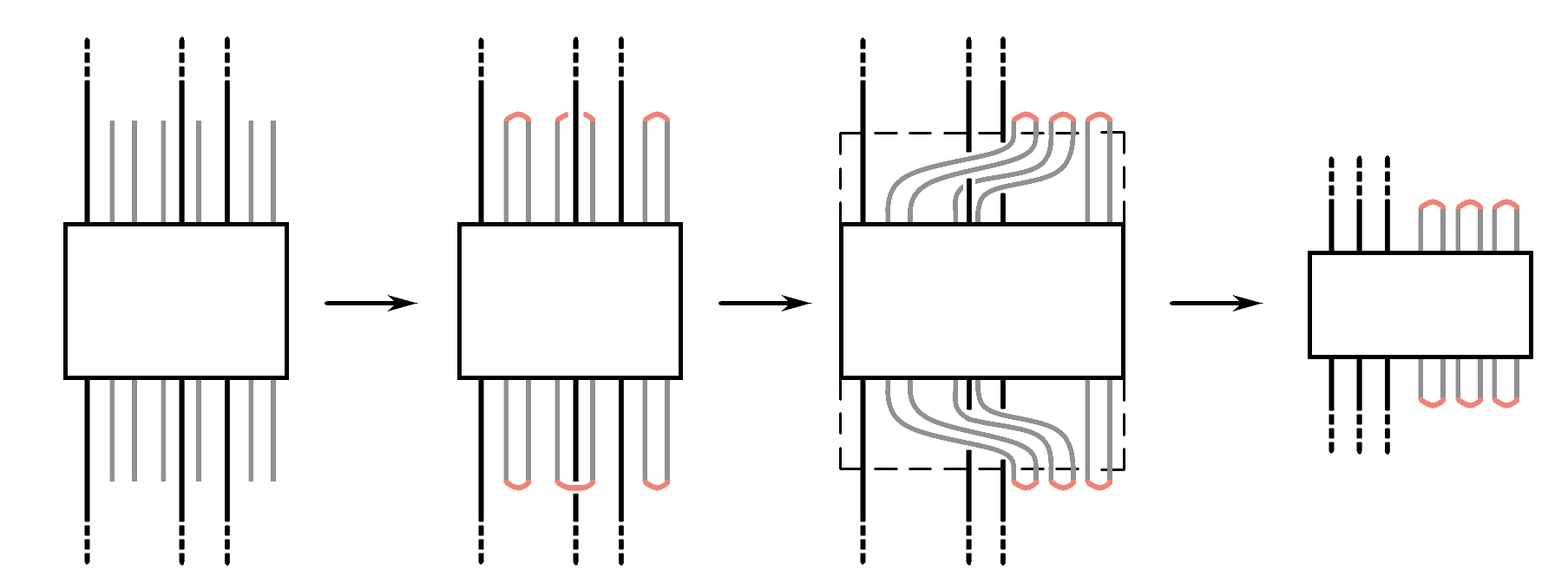}
    \caption{Exemplifying the process of parting a geometric mixed plat to an algebraic  mixed plat.}
    \label{fig:spike_parting}
\end{figure}

\section{The mixed Hilden braid group}\label{Sec:Mixed_hilden_group}

As in the classical case of $\mathbb{R}^3$, we want to define the Hilden braid group also in the setting of a handlebody. We will do that adapting Definition~\ref{def:Hilden_group} to the context of the surface braid group of $D_g$ (recall Remark~\ref{rmk:braid_surface_disk}) resp. the mixed braid group.

\begin{definition}  \rm 
 Let \( \mathcal{C} = \{C_1,\ldots,C_{2m}\} \) be a set of $2m$ points in $D_g$, which w.l.o.g. lie in a disc $A$ on the right-hand side of $D_g$. The {\it Hilden group of the handlebody $H_g$}, denoted \(K_{D_g,2m}\), is defined as the subgroup of $\mathcal{B}_{D_g,2m}$ generated by the equivalence classes of homeomorphisms of \(D_g \times I\)  leaving the set of closure arcs invariant on its top/bottom boundary.   

Furthermore, the {\it mixed Hilden group} \(K_{g,2m}\) is defined as the subgroup of the mixed braid group \(\mathcal{B}_{g,2m}\) generated by the equivalence classes of homeomorphisms of \(\mathbb{R}^3_+\)  leaving the set of bottom (resp. top) plat closure arcs invariant on its boundary. 
\end{definition}

\begin{remark}  \rm 
Clearly, an isomorphism between  $\mathcal{B}_{D_g,2m}$ and \(\mathcal{B}_{g,2m}\) induces an isomorphism between \(K_{D_g,2m}\) and \(K_{g,2m}\).  Hence, and following the idea of Remark \ref{rmk:braid_surface_disk}, the mixed Hilden braid group can be viewed either as the Hilden subgroup of the mixed braid group of the handlebody or the Hilden subgroup of the surface braid group of \(D_g\).

\end{remark} 

\begin{theorem}\label{theo:hilden_mixed}
A set of generators for \(K_{g,2m}\) is \[\{\sigma_{1}, \  \lambda_1 = \sigma_{2} \sigma_1^2 \sigma_{2}, \  \mu_i = \sigma_{2i} \sigma_{2i-1} \sigma_{2i+1} \sigma_{2i}, \ 1\leq i \leq m-1; \  \tau_{i,2j-1}, \ 1\leq i\leq g, \, 1\leq j\leq m\},\] where: 
\begin{align*}
\tau_{i,1} &= \alpha_i \, \sigma_{1} \alpha_i \sigma_{1} \\
\tau_{i,2j-1} &= \sigma_{2j-2} \sigma_{2j-3} \sigma_{2j-1} \sigma_{2j-2} \dots \sigma_{2} \sigma_{1} \sigma_{3} \sigma_{2} \alpha_i \sigma_{1} \, \alpha_i \, \sigma_{1}  \sigma_{2}^{-1} \sigma_{3}^{-1} \sigma_{1}^{-1} \sigma_{2}^{-1} \dots \sigma_{2j-2}^{-1} \sigma_{2j-1}^{-1} \sigma_{2j-3}^{-1} \sigma_{2j-2}^{-1}
\end{align*}
The generators \(\sigma_{1}, \lambda_1, \mu_i, \tau_{i,2j-1} \)  shall be referred to as `mixed Hilden generators'.
\end{theorem} 
\noindent View Fig.~\ref{fig:gen_hilden_mixed} for the generators $\sigma_{1}, \lambda_1$ and $\mu_i$ and Fig.~\ref{fig:generator_extra_hilden} for the generators $\tau_{i,2j-1}$.

\begin{figure}[H]
    \centering
    \includegraphics[width = .85\textwidth]{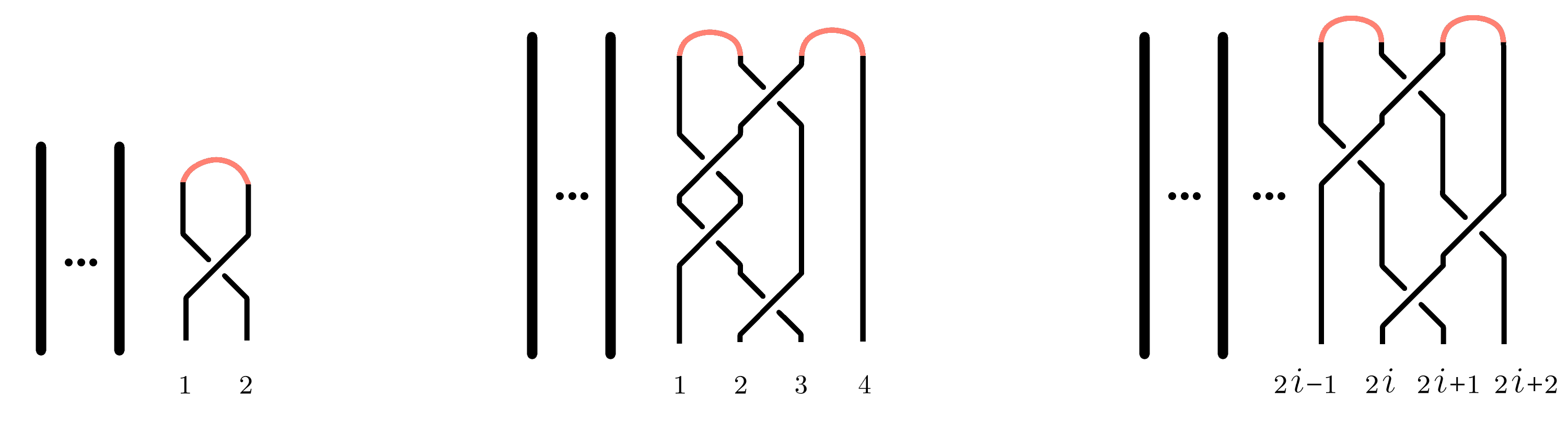}
    \caption{The mixed Hilden generators \(\sigma_1, \lambda_1\) and \(\mu_i\).}
    \label{fig:gen_hilden_mixed}
\end{figure}

\begin{figure}[H]
    \centering
    \includegraphics[width = .55\textwidth]{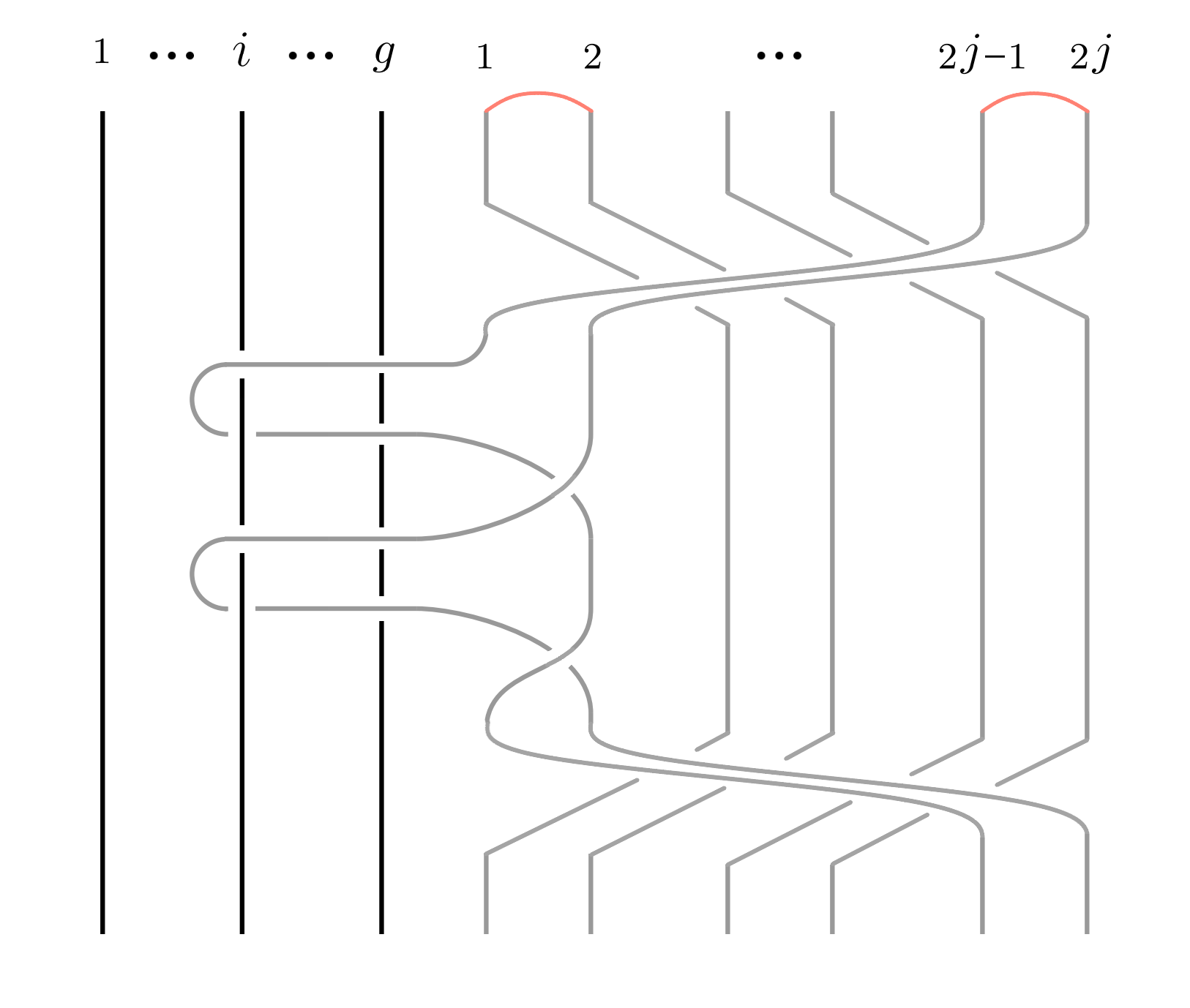}
    \caption{The mixed Hilden generator \(\tau_{i,2j-1}\).}
    \label{fig:generator_extra_hilden}
\end{figure}

\begin{proof}
Since \(\mathcal{B}_{g,2m}\) is a subgroup of \(\mathcal{B}_{g+2m}\),  the idea is to make use of Proposition~\ref{prop:hilden_generators}. For this, we need to also argue that \(K_{g,2m}\) is a subgroup of \(K_{g+2m}\). If {\it \(g\) is even} then \(K_{g+2m}\) is well-defined and, hence, \(K_{g,2m}\) coincides with the subgroup of \(K_{g+2m}\) keeping fixed the first \(g\) strands. In the case of {\it \(g\) odd} this does not make sense, because the mixed Hilden subgroup is defined only for an even number of moving strands. For this reason we will introduce in this case a \emph{dummy strand} with number zero on the left of the first  strand of any element of \(K_{g+2m}\) as well as on the left of the first fixed strand of any element of  \(K_{g,2m}\). The dummy strand does not interact with any other strand of the braid. In this way, every braid will have \( (1+g+2m)\) strands, and we can define then the dummied \(K^\prime_{g,2m}\) to be the Hilden subgroup of the dummied \(K^\prime_{g+2m}\), for \(g\) odd. With the above trick, we can  always assume that \(K_{g+2m}\) is well-defined and \(K_{g,2m}\) is a subgroup of \(K_{g+2m}\).

In using Proposition~\ref{prop:hilden_generators}, we will start by examining the full set of generators for the group \(K_{g+2m}\). 
Then, the generators of \(K_{g,2m}\) will be the ones inherited by \(K_{g+2m}\), but with different enumeration and some small adjustments for both cases for \(g\), even or odd. First of all, since the first \(g\) strands of an element of \(K_{g,2m}\) are fixed, the  \(\sigma_i\) generators involving the first $g$ strands need to be eliminated.  Likewise
for the generators \(\mu_i\) with indices smaller than \(g\), since the first $g$ strands do not participate in the plat closure. Second, we can divide \(\rho_{ij}\) and \(\omega_{ij}\) generators into three families depending on their indices: 
\begin{itemize}
    \item \(2i,2j \leq g\): these generators need to be eliminated, since they move fixed strands; 
    \item \(2i\leq g, \ 2j>g\): these generators still exist in \(K_{g,2m}\), but we need to write them using generators of the mixed braid group \(B_{g,2m}\) and renumber the indices; 
    \item \(2i,2j >g\): these generators remain, but a renumbering of the indices is needed.
\end{itemize}

\begin{figure}[H]
    \centering
    \includegraphics[width = \textwidth]{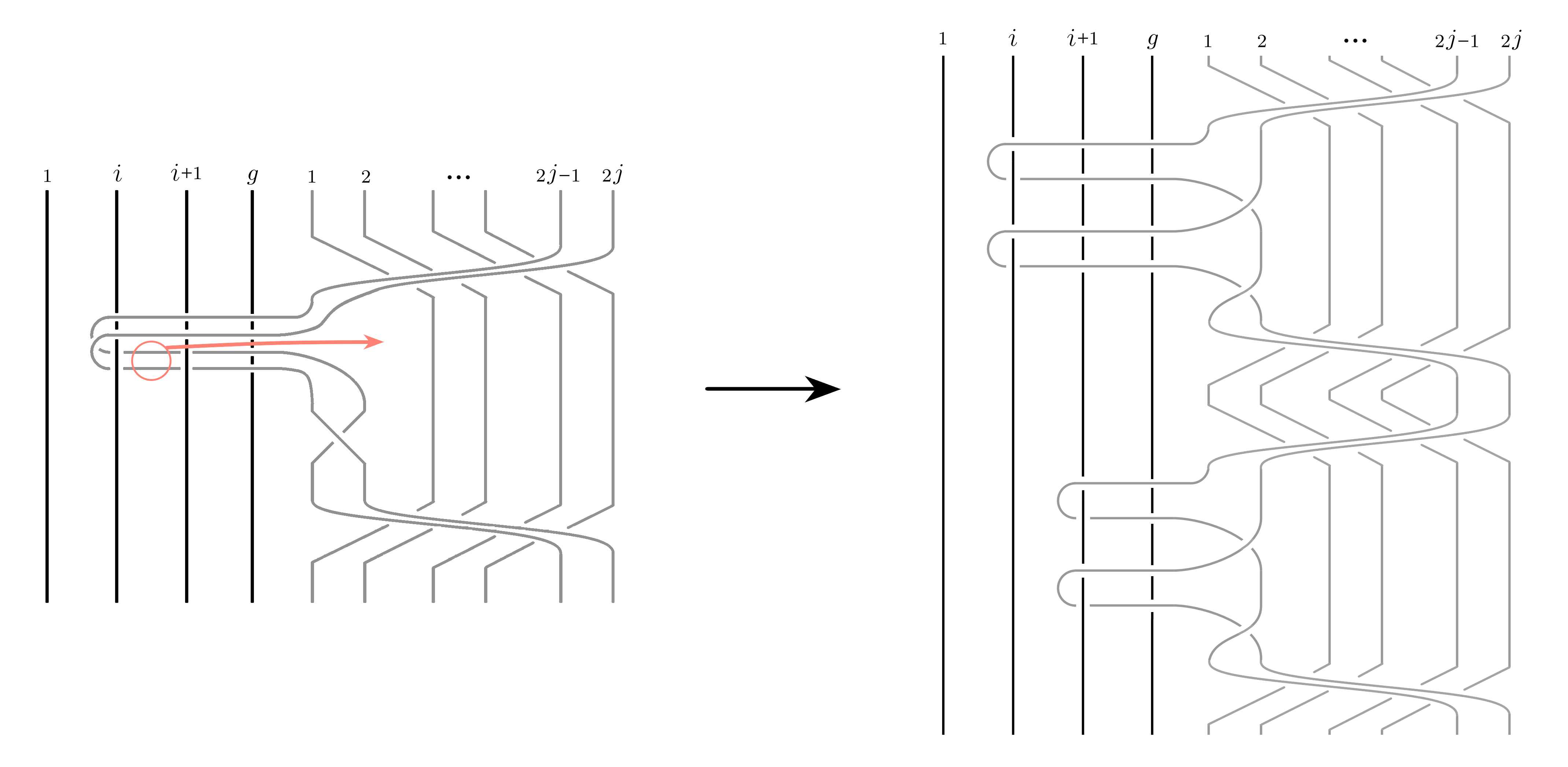}
    \caption{Finding the expression for \(\omega_{ij}\) as a product of \(\tau\) generators.}
    \label{fig:generator_extra_hilden_2}
\end{figure}

Now that the set of generators is found, we need to renumber their indices  keeping in mind that the \((i+g)\)-th strand in \(\mathcal{B}_{g+2m}\) is numbered \(i\) in \(\mathcal{B}_{g,2m}\). We rename the \(\rho_{ij}\)  generators of the second family to  \(\tau_{i,2j-1}\:\) for \( i = 1, \dots, g,\: j = 1, \dots, m\), that takes the strands numbered \(2j-1, 2j\) around the \(i\)-th fixed strand. Then an \(\omega_{ij}\) generator of the second family can be expressed using the \(\tau_{k,l}\)-type generators. Indeed, Figure \ref{fig:generator_extra_hilden_2} illustrates the expression of the  \(\omega_{ij}\) generator as the product   \(\tau_{i,2j-1} \cdot \tau_{i+1, 2j-1}\). 

Note that, in the case of \(g\) odd, we will have also the definition of \(\tau_{0,2j-1}, j = 1, \dots, m\). In this way we found a way to represent generators \(\rho_{ij}\) and \(\omega_{ij}\) using the generators of \(\mathcal{B}_{g,2m}\). 
 At this point, if \(g\) is even, the construction is completed. If \(g\) is odd, we have to remove the dummy strand, numbered zero. To perform this transformation we follow the natural inclusion \(K_{g+1,2m} \rightarrow K_{g,2m}\), where the leftmost strand is forgotten. This means that the generators \(\tau_{0,2j-1}\) are now trivial elements of the group. 

Last but not least, using the relations found by Birman in \cite{birman1976stable}, which led to a minimal set of generators, we can introduce the generator \(\lambda_1\) and omit from the set of generators the third family of \(\rho_{ij}\)'s and \(\omega_{ij}\)'s, the ones with \(i,j >g\), since they can be represented using the other generators, as in the classical case. 
\end{proof}

\begin{remark}\rm
It is important to note that not all \(\tau_{i, 2j-1}\) generators are needed: in fact, for \(j \neq 1\) it is possible to obtain \(\tau_{i,2j-1}\) by composing \(\mu_k\) generators with \(\tau_{i,1}\). 
\end{remark}


\section{The equivalence between plat closed mixed braids}\label{Equivalence}

Before going on, let us recall from  \cite{birman1976stable} the definition of stabilization move for plat closed braids, adapted to the case of a link in the handlebody $H_g$, or equivalently, a mixed link. 

\begin{definition} \label{def:mixedstabilization} \rm 
Let \(a, b, c\) be consecutive vertices of our link in $H_g$ (resp. a mixed link $I_g\cup L$) in polygonal form. Let \(b'\) (resp. \(b''\)) be the projections of \(b\) onto \(D_g \times \{0\}\), \(D_g \times \{1\}\) (resp.  on the top and bottom lines containing the endpoints), and \(b_1, b_2\) the intersection points of a sufficiently small neighbourhood of \(b\) with the segments \([a, b]\) and \( [b,c]\). The \textit{stabilization move}  is the substitution of \([a, b]\) and \( [b,c]\) with the segments \([a, b_1], [b_1, b'], [b', b''], [b'', b_2]\) and \( [b_2, c]\). View Fig.~\ref{fig:stabilization}.  
\end{definition}
By definition, a stabilization move is an isotopy. 
\begin{figure}[H]
    \centering
    \includegraphics[width = .7\textwidth]{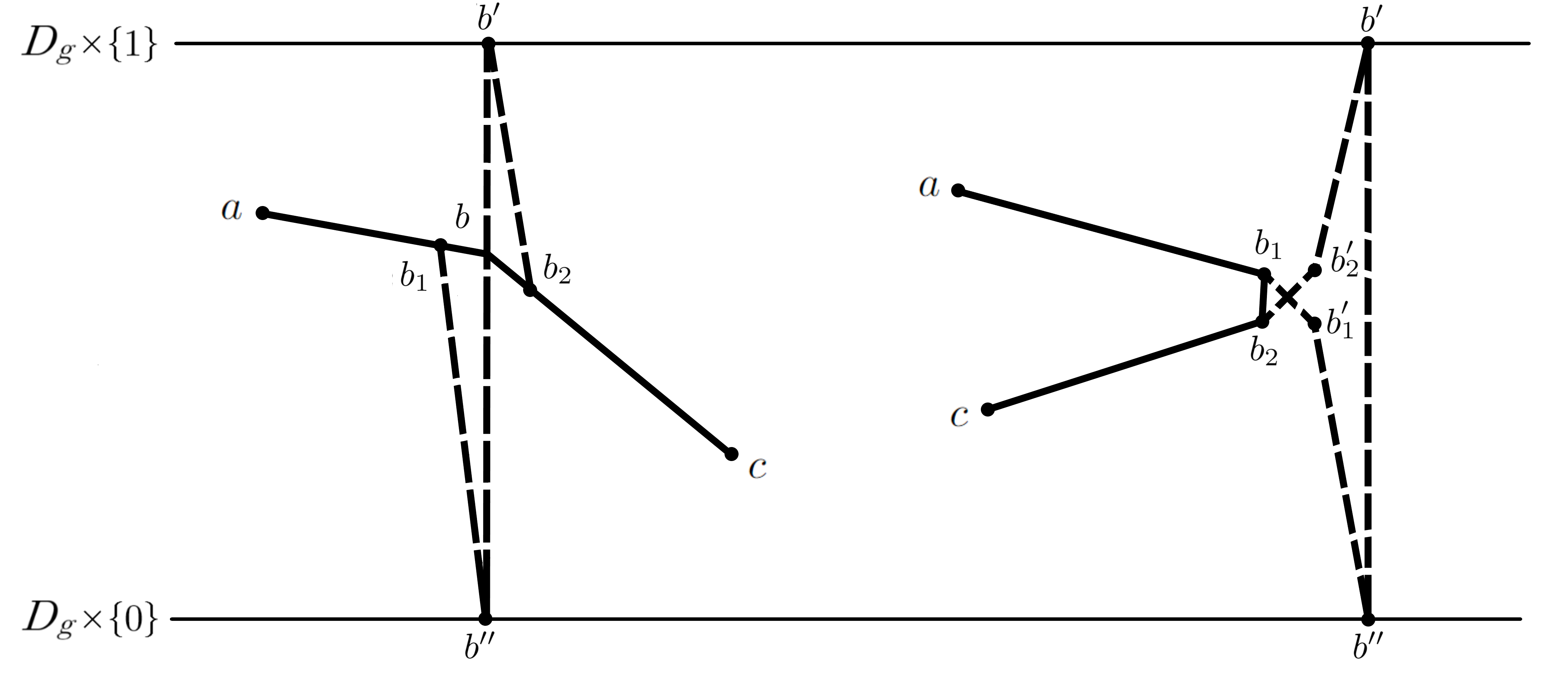}
    \caption{The two forms of the stabilization move for $H_g$.}
    \label{fig:stabilization}
\end{figure} 

Let us now consider two plat closed mixed braids \(B, B' \in \mathcal{B}_{g, 2n}\) such that one can be obtained by the other by only applying elements of the mixed Hilden group \(K_{g,2n}\). Then, since all the generators of \(K_{g,2n}\) represent local isotopies for a plat closed mixed braid, the two braids represent the same link isotopy class via plat closure. To formulate and prove the converse statement, a previous result of Birman  helps us: 

\begin{lemma}(\cite[Lemma 8]{birman1976stable}) \label{Birman_lemma}
Let \(L\) and \(L'\) be two isotopic polygonal links in general position, i.e. with no horizontal edges and such that each vertex has nothing above or underneath it in a small neighbourhood of it. Then it is possible to connect the two links by a finite sequence of spike moves and stabilization moves. 
\end{lemma}
The lemma is proved in the case of \(\mathbb{R}^3\), by examining the possible $\Delta$-moves. The only difference when adapting to $H_g$ is that the $\Delta$-moves miss the vertical lines of the punctures. But this fact does not present any obstacle, since the spike moves and the stabilization moves are the same moves as in \(\mathbb{R}^3\) and they also  miss the vertical lines (the same arguments carry through also to the setting of  mixed links). The spike move for the case of $H_g$ is defined in Definition~\ref{def:mixedspike}, while the stabilization move is defined just above, in Definition~\ref{def:mixedstabilization}. Namely we have: 

\begin{lemma} \label{mixed_Birman_lemma}
Let \(L\) and \(L'\) be two isotopic polygonal links in the handlebody $H_g$, resp. mixed links, in general position, i.e. with no horizontal edges and such that each vertex has nothing above or underneath it in a small neighbourhood of it. Then it is possible to connect the two links by a finite sequence of spike moves and stabilization moves. 
\end{lemma}

Before formulating and proving the plat equivalence for the handlebody we will introduce a special case of a stabilization move, the $st_k$ move.

\begin{definition}\label{def:stk} \rm
An {\it $st_k$ move}  is a stabilization move   between two algebraic mixed braids \(B, B' \in \cup_{n\in \mathbb{N}} \mathcal{B}_{g,2n}\), defined as:
\[\begin{array}{cccl}
  st_k : & \mathcal{B}_{g,2n} & \rightarrow & \mathcal{B}_{g, 2n+2} \\
   & B & \mapsto & B \sigma_{2k}^{-1} \sigma_{2k+1}^{-1} \dots \sigma_{2n-1}^{-1} \sigma_{2n} \sigma_{2n-1} \dots \sigma_{2k+1} \sigma_{2k}
\end{array}\]

\noindent This move, depicted abstractly in Fig.~\ref{fig:st_k_move},  represents a special case of a stabilization move between algebraic mixed braids, taking place at the bottom of an even  numbered strand $2k$, with a positive crossing. 
\end{definition}

\begin{figure}
    \centering
    \includegraphics[width = .85\textwidth]{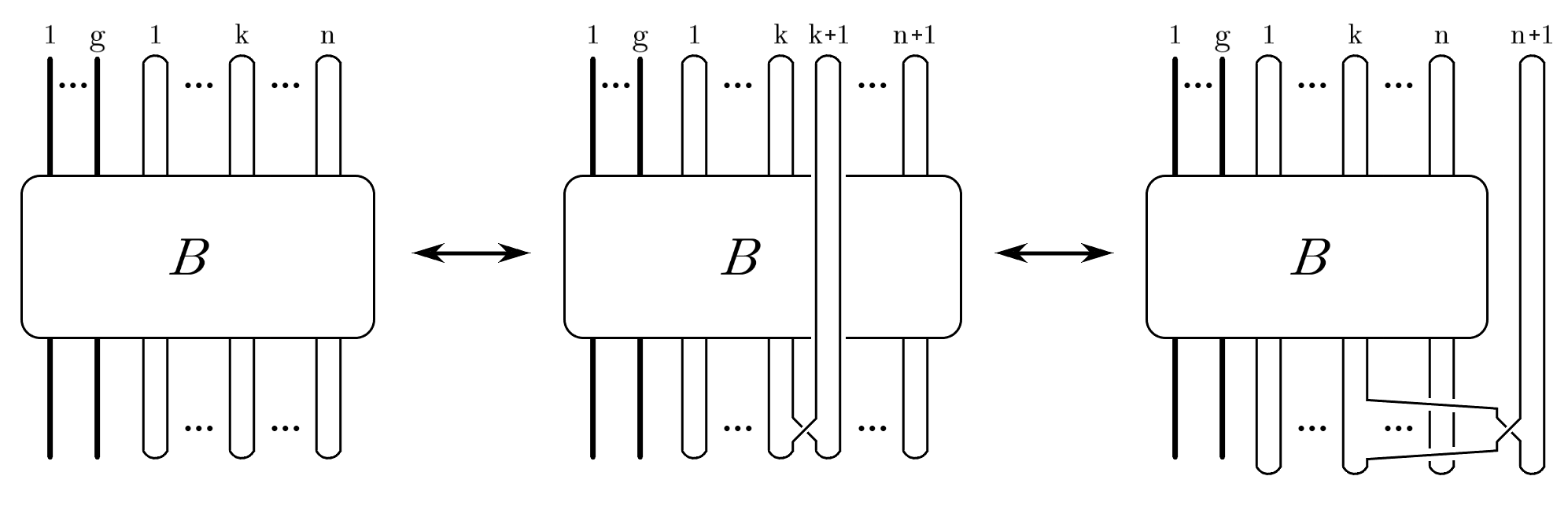}
    \caption{The \(st_k\) move.}
    \label{fig:st_k_move}
\end{figure}

\begin{theorem}\label{teo:main_equivalence}
Two braids \(B, B' \in \cup_{n\in \mathbb{N}} \mathcal{B}_{g,2n}\) determine, via plat closure, equivalent links in the handlebody \(H_g\) of genus $g$ if and only if they are connected by braid relations and a finite sequence of the following moves: 
\[\begin{array}{rrcccl}
h_{\sigma}:&        \sigma_1 B & \longleftrightarrow &B& \longleftrightarrow & B \sigma_1, \\[\smallskipamount]
h_{\lambda}:&       \sigma_2 \sigma_1^2 \sigma_2 B & \longleftrightarrow &B& \longleftrightarrow & B \sigma_2 \sigma_1^2 \sigma_2,  \\[\smallskipamount]
h_{\mu}:&          \sigma_{2i} \sigma_{2i-1} \sigma_{2i+1} \sigma_{2i} B & \longleftrightarrow &B& \longleftrightarrow & B \sigma_{2i} \sigma_{2i-1} \sigma_{2i+1} \sigma_{2i},  \ \  i = 1, \dots, n \\[\smallskipamount]
h_{\tau}:&          \alpha_j \sigma_{1} \alpha_j \sigma_{1} B & \longleftrightarrow &B& \longleftrightarrow & B \alpha_j \sigma_{1} \alpha_j \sigma_{1},  \ \  j = 1, \dots, g\\[\smallskipamount]
st_n:&              & &B& \longleftrightarrow & B \sigma_{2n} 
\end{array}
\]
\end{theorem}

\noindent The moves $h_\sigma, h_\lambda, h_\mu, h_\tau$ reflect the generators of the mixed Hilden braid group (recall Figs.~\ref{fig:gen_hilden_mixed} and~\ref{fig:generator_extra_hilden}), so do not alter the number of strands, while the move $st_n$, defined above, realizes the stabilization move that may take place at the bottom right of the last strand numbered \(2n\). 
\begin{proof}
The one direction of the proof is obvious. Given now a link \(L\) in $H_g$ in standard position, we obtain by Theorem \ref{teo_braiding_plat_hand} an isotopic link (a plat) with \(n\) points of intersection with \(D_g \times \{0\}\), say \(Y^u = \{Y_1^u, \dots, Y^u_n\}\), and \(n\) points of intersection with \(D_g \times \{1\}\), say \(Y^l = \{Y_1^l, \dots, Y^l_n\}\).
 In order to find an element \(B \in \mathcal{B}_{g,2n}\) such that its plat closure is isotopic \(L\),  we have to perform  plat parting to the geometric mixed plat.  The plat parting comprises a sequence of specific spike moves and is described in the proof of Theorem~\ref{teo:lemma_algebraic_handleb}. 
 
 For each parting move there is a choice of performing it by pulling over or under strands to the right. 
 For proving Theorem~\ref{teo:main_equivalence} we will argue that these choices are included in the equivalence relation.  
Indeed, for a parting choose two paths \(p\) and \(p'\) in the configuration space of \(n\) points on \(D_g\) connecting, respectively, the two sets of boundary points \(Y^u\) and \(Y^l\) to a set of points \(\mathcal{C} = \{C_1, \dots, C_n\}\), where \(C_i\) is an internal point of the disk \(A\) (recall Fig.~\ref{fig:handlebody_disk}). This choice of paths is not unique: if \(p\) and \(q\) are two choices for \(Y^l\), the composition of \(p\) and the inverse of \(q\) is an element in the fundamental group of the configuration space of \(\mathcal{C}\), that could be realized as a composition of elements of the mixed Hilden braid group. The same construction holds for the set of upper boundary points. This means that any two element of \(\mathcal{B}_{g,2n}\) obtained as above for the same link \(L\) in standard position are connected by moves \(h_{\sigma}, h_{\lambda}, h_{\mu}, h_{\tau}\). 
It would be illuminating to consider the special case of the choice of the plat parting (recall Fig.~\ref{fig:generator_extra_hilden}) by pulling a pair of consecutive strands {\it over} a fixed strand versus the choice of pulling them {\it under} the same fixed strand. Such a choice difference shows in the final algebraic braid words as a mixed Hilden generator \(\tau_{i,2j-1}\) (recall Fig.~\ref{fig:generator_extra_hilden}).

By Lemma~\ref{Birman_lemma}, it is now sufficient to show how stabilization and spike moves change the algebraic mixed braid representative \(B\) of a link \(\overline{B}\), the plat closure of \(B\). 
Recall that a stabilization move only adds a trivial loop to \(\overline{B}\) at any point. Using the sequence of moves  depicted in Fig.~\ref{fig:stabilization_to_standard_anywhere} and sliding the stabilization, it is always possible to consider it taking place at the bottom right of an even strand, up to spike moves. Then, by spike moves, it is always possible to consider it happening on the top of the braid. After sliding the newly generated arcs to the right of the braid, we will obtain something as depicted in the rightmost part of Fig.~\ref{fig:st_k_move}. Moreover, as demonstrated in Fig.~\ref{fig:stabilization_to_standard_crossing} the move can be assumed to have a positive crossing only, as in move $st_k$ in Definition~\ref{def:stk}. It is straightforward to check that this move is the \(st_k\) move, if the stabilization is happening on the strand numbered \(2k\). 

\begin{figure}[H]
    \centering
    \includegraphics[width = \textwidth]{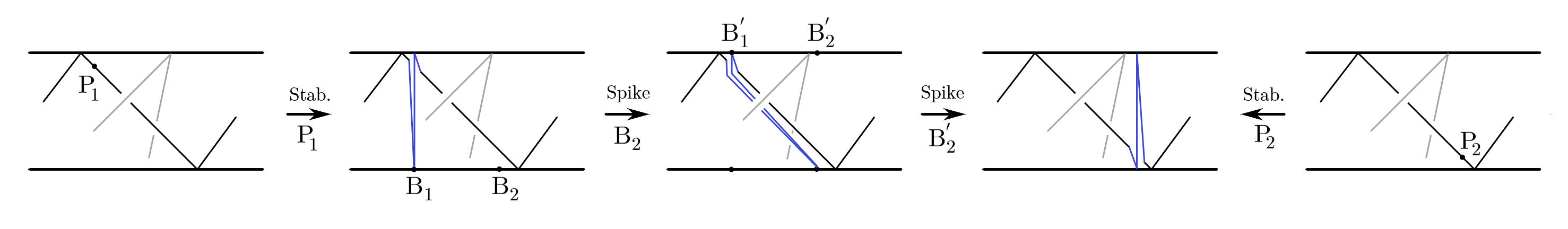}
    \caption{The procedure to slide a stabilization move to the bottom right part of an even numbered strand.}
    \label{fig:stabilization_to_standard_anywhere}
\end{figure}

\begin{figure}[H]
    \centering
    \includegraphics[width = \textwidth]{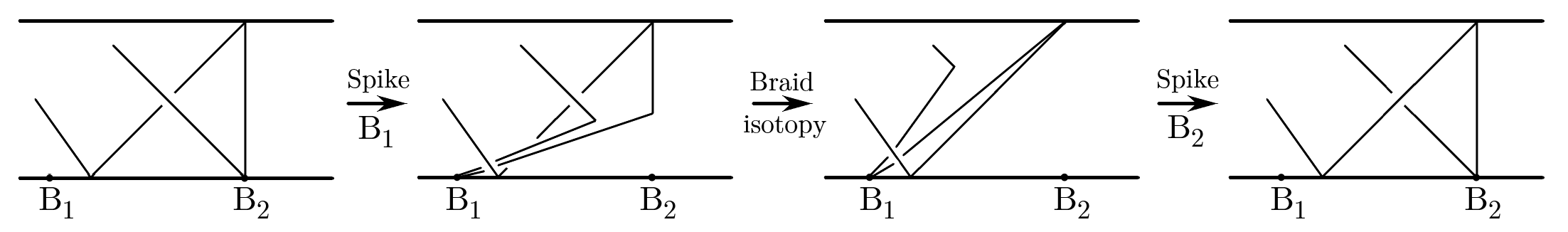}
    \caption{The procedure to turn a negative crossing into a positive one.}
    \label{fig:stabilization_to_standard_crossing}
\end{figure}

We shall now argue that in the equivalence theorem we only need the \(st_n\) move. Indeed, note that before performing an \(st_k\) move, one can bring the pair of strands (\(2k-1, 2k\)) to the far right of the braid by using \(h_\mu\) moves. After performing the stabilization \(st_n\), all the four strands \(2n-1, 2n, 2n+1, 2n+2\) can be brought again to the previous positions \(2k-1, 2k, 2k+1\) and \(2k+2\) by means of \(h_\mu\) moves. The result, after braid isotopy, is equivalent to the application of the \(st_k\) move to the initial braid. For the abstraction of the whole process, the reader is referred to Fig.~\ref{fig:stn_only}. 

\begin{figure}[H]
    \centering
    \includegraphics[width = \textwidth]{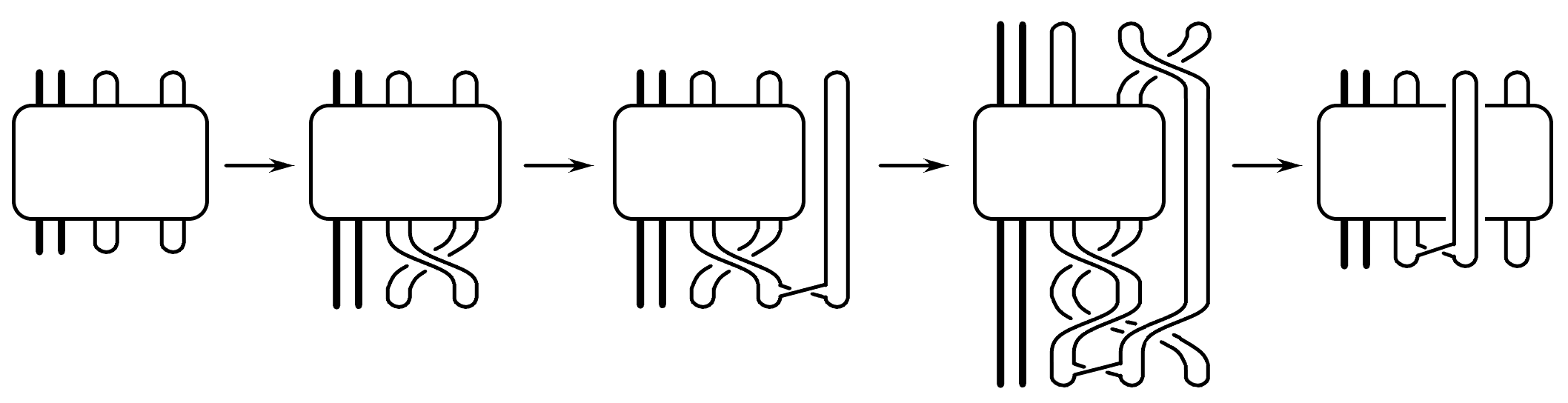}
    \caption{The process of performing \(st_k\) by means of \(h_\mu\), \(st_n\) and braid isotopy moves.}
    \label{fig:stn_only}
\end{figure}

 Suppose, now, that starting from an algebraic mixed plat we perform a stabilization move on a loop generator, obtaining a geometric mixed plat.  In this case we perform the parting in a way compatible to the layering of the stabilization move, so as to obtain an algebraic stabilization move, which is analyzed above. See Fig.~\ref{fig:geometric_stabilization}.

\begin{figure}[H]
    \centering
    \includegraphics[width = .5\textwidth]{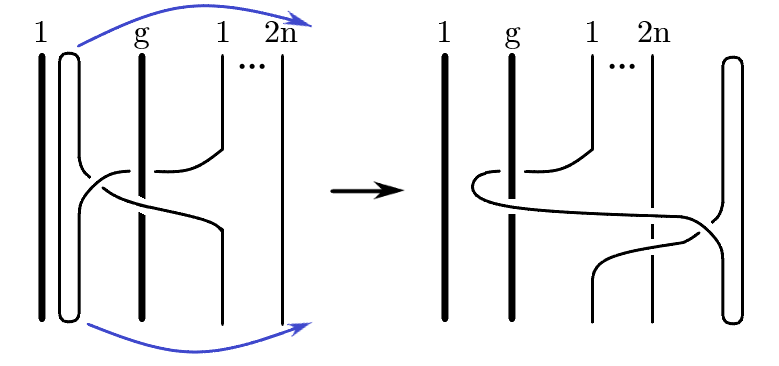}
    \caption{Transferring a geometric stabilization move to an algebraic one.}
    \label{fig:geometric_stabilization}
\end{figure}

Let us, finally, consider  a spike move taking place in \(\overline{B}\), obtaining \(\overline{B'}\). Suppose that the spike move brings the upper boundary point \(C_i\) into another point in the disk \(A\), say \(C_i'\). Following Fig.~\ref{fig:spike_move}, where \(C_i\) coincides with \(b_{i+1}\) and \(C_i'\) with \(b_{i+1}'\), we can find \(t_0 \in [0,1]\) such that \([b_i, b_{i+2}]\) lie in \(D_g \times \{t_0\}\). Considering the middle point of \([b_i, b_{i+2}]\) as \(C_i^*\), we can express the spike move as a path bringing \(C_i^*\) to \(C_i'\), following the median of the abstract triangle \([b_i, C_i', b_{i+2}]\). 
The plane \(D_g \times \{t_0\}\) divides the braid into two parts, say \(B_1\) the upper part and \(B_2\) the lower part. Now consider \(B_1 B_1^{-1}\), where the two spike moves connect together onto the base \([b_i, b_{i+2}]\). Since \(B_1 B_1^{-1}\) is a trivial element in \(B_{g,2n}\), it can be represented via elements of the Hilden braid group. In particular, the two spike moves bound a disk and can be represented as a product of the moves \(h_{\sigma}, h_{\lambda}, h_{\mu}, h_{\tau}\). This means that also the initial spike move could be represented via a combination of these four moves. The symmetric construction holds in the case of a spike move on a lower boundary point. 

Suppose, finally, that starting from an algebraic mixed plat we perform a spike move toward the left part of the braid, obtaining a geometric mixed plat.  In this case, we just use for the parting the inverse move.

The proof of the Theorem is now complete.
\end{proof} 

 In the spirit of Birman's proof of Theorem~\ref{Thm:birman_original}, 
  Theorem~\ref{teo:main_equivalence} can be re-stated equivalently as follows:

\begin{theorem}
Let \(L_i, i=1,2\) be two links in a handlebody \(H_g\), represented by the plat closure of \(B_i \in \mathcal{B}_{g, 2n_i}, i=1,2\). Then the two links are isotopic if and only if there exists an integer \(t \geq \max{n_1, n_2}\) such that, for each \(m \geq t\) the two elements: 
\[B_i' = B_i \sigma_{2n_i} \sigma_{2n_i +2} \dots \sigma_{2m-2} \in \mathcal{B}_{g, 2m}, \ \ i=1,2\]
are in the same double coset of \(\mathcal{B}_{g, 2m}\)  modulo the mixed Hilden subgroup \(K_{g, 2m}\). 
\end{theorem}

\begin{remark} \rm 
In \cite{cattabriga2018markov} Cattabriga and Gabrovšek formulate and prove the plat equivalence for links in a thickened surface. Following  their work and suppressing the $b$-type loop generators (looping around the handles), one would obtain the statement of Theorem~\ref{theo:hilden_mixed}. 
\end{remark}

\bibliographystyle{plain}
\bibliography{sample}

\end{document}